\documentclass[10pt]{article}
\usepackage{amsmath,amssymb,amsthm,amscd}

\numberwithin{equation}{section}

\DeclareMathOperator{\vol}{vol}
\DeclareMathOperator{\Hess}{Hess}

\DeclareMathOperator{\diam}{diam}
\DeclareMathOperator{\V}{V}
\DeclareMathOperator{\Isop}{Isop}

\def\CC{{\mathbb C}}

\newtheorem{prop}{Proposition}[section]
\newtheorem{theo}[prop]{Theorem}
\newtheorem{lemm}[prop]{Lemma}
\newtheorem{clai}[prop]{Claim}
\newtheorem{coro}[prop]{Corollary}
\newtheorem{rema}[prop]{Remark}

\newtheorem{conj}[prop]{Conjecture}
\def\begeq{\begin{equation}}
\def\endeq{\end{equation}}

\title{Regularity of K\"ahler-Ricci flows on Fano manifolds}

\author{Gang Tian\thanks{The first author is supported by NSF grants. Email: tian@math.princeton.edu}\\ SMS and BICMR, Peking University, Beijing 100871, China\\Department of Mathematics,
Princeton University, NJ 08544, USA
\and
Zhenlei Zhang\thanks{The second author is supported by a grant of Beijing MCE 11224010007 and NSFC 13210010022. Email: zhleigo@aliyun.com}\\
School of Mathematics, Capital Normal University, Beijing 100048, China}

\date{}

\begin{document}

\maketitle

\abstract{In this paper, we will establish a regularity theory for the K\"ahler-Ricci flow on Fano $n$-manifolds with Ricci curvature bounded in $L^p$-norm for some $p > n$. Using this regularity theory, we will also solve a long-standing conjecture for dimension 3. As an application, we give a new proof of the Yau-Tian-Donaldson conjecture for Fano 3-manifolds. The results have been announced in \cite{TiZh12b}.}

\tableofcontents

\section{Introduction}

This is the first part of a series of papers on the long-time behavior of K\"ahler-Ricci flows on Fano manifolds. We will solve a long-standing conjecture
in low dimensional case.

Let $M$ be a Fano $n$-manifold. Consider the normalized K\"ahler-Ricci flow:
\begin{equation}\label{KRF}
\frac{\partial g}{\partial t}\,=\,g\,-\,{\rm Ric}(g).
\end{equation}
It was proved in \cite{Ca85} that \eqref{KRF} has a global solution $g(t)$ in the case that $g(0)=g_0$ has canonical K\"ahler class, i.e., $2\pi c_1(M)$ as its
K\"ahler class.
The main problem is to understand the limit of $g(t)$ as $t$ tends to $\infty$.
A desirable picture for the limit is given in the following folklore conjecture \footnote{It has been often referred
as the Hamilton-Tian conjecture in literatures, e.g., in \cite{Pe02}. Also see \cite{Ti97} for a formulation of this conjecture.}

\begin{conj}[\cite{Ti97}]
\label{conj:HT}
$(M,g(t))$ converges (at least along a subsequence) to a shrinking K\"ahler-Ricci soliton with mild singularities.
\end{conj}

Here,``mild singularitie'' may be understood in two ways: (i) A singular set of codimension at least $4$, and (ii) a singular set of a normal variety. The first interpretation concerns the differential geometric part of the problem where the convergence is taken in the Gromov-Hausdorff topology, while in the second interpretation the spaces $(M,g(t))$ converge as algebraic varieties in some projective space. By extending the partial $C^0$-estimate conjecture \cite{Ti10} to the K\"ahler-Ricci flow, one can show that these two approaches are actually equivalent (see Theorem \ref{regularity:2} below and Section 5).

This conjecture implies another famous conjecture, the Yau-Tian-Donaldson conjecture, in the case of Fano manifolds. The Yau-Tian-Donaldson conjecture states that a Fano manifold $M$ admits a K\"ahler-Einstein metric if and only if it is K-stable. The necessary part of the conjecture is proved by the first named author in \cite{Ti97}. Last Fall, the first named author gave a proof for the sufficient part (see \cite{Ti12}) by establishing the partial $C^0$-estimate for conic K\"ahler-Einstein metrics.
Another proof was given in \cite{ChDoSu1, ChDoSu2, ChDoSu3}. As we will see in the sections below, the essential step in the resolution of conjecture (\ref{conj:HT}), as
for proving Yau-Tian-Donaldson conjecture, is the Cheeger-Gromov convergence of the K\"ahler-Ricci flow.

Let us recall some known facts on the K\"ahler-Ricci flow. By the noncollapsing result of Perelman \cite{Pe02}, there is a positive constant $\kappa$ depending only on $g_0$ such that
\begin{equation}\label{volume noncollaping:0}
\vol_{g(t)}(B_{g(t)}(x,r))\geq\kappa r^{2n},\hspace{0.5cm}\forall t\geq 0, r\leq 1.
\end{equation}
Also by Perelman, the diameter and scalar curvature of $g(t)$ are uniformly bounded (see \cite{SeTi08} for a proof). Since the volume stays the same along the K\"ahler-Ricci flow, the noncollapsing property (\ref{volume noncollaping:0}) implies that for any sequence $t_i \rightarrow\infty$,
by taking a subsequence if necessary, $(M,g(t_i))$ converge to a limiting length metric space $(M_\infty,d)$ in the Gromov-Hausdorff topology:
\begin{equation}\label{e13}
(M,g(t_i))\stackrel{d_{GH}}{\longrightarrow}(M_\infty,d).
\end{equation}
The question remained is the regularity of the limit $M_\infty$. In the case of Del-Pezzo surfaces, or in the higher dimension with additional assumption of uniformly bounded Ricci curvature or Bakry-\'Emery-Ricci curvature, the regularity of $M_\infty$ has been checked, cf.  \cite{Se05}, \cite{ChWa12} and \cite{TiZh12}.
If $M$ admits K\"ahler-Einstein metrics in a prior, Perelman first claimed that the K\"ahler-Ricci flow converges to a smooth K\"ahler-Einstein metric and showed a few
crucial estimates towards his proof. Tian-Zhu gave a proof of this and generalized this to the case of K\"ahler-Ricci solitons under the assumption that the metric is invariant by the holomorphic vector field of the Ricci soliton \cite{TiZhu07}; see also \cite{TiZhu13, TZZZ}.

The main result of this paper is the following

\begin{theo}\label{regularity:1}
Let $(M,g(t))$, $t_i$ and $(M_\infty, d)$ be given as above. Suppose that for some uniform constants $p>n$ and $\Lambda< \infty$,
\begin{equation}\label{Ricci:Lp0}
\int_M|Ric(g(t))|^pdv_{g(t)}\,\leq\,\Lambda.
\end{equation}
Then the limit $M_\infty$ is smooth outside a closed subset $\mathcal{S}$ of (real) codimension $\geq 4$ and $d$ is induced by a smooth K\"ahler-Ricci soliton $g_\infty$ on $M_\infty\backslash\mathcal{S}$. Moreover, $g(t_i)$ converge to $g_\infty$ in the $C^\infty$-topology outside $\mathcal{S}$.\footnote{The convergence with these properties is also referred as the convergence in the Cheeger-Gromov topology, see \cite{Ti90} for instance.}
\end{theo}

\begin{rema}
In view of the main result in \cite{TZZZ}, one should be able to prove that under the assumption of Theorem \ref{regularity:1}, the K\"ahler-Ricci flow $g(t)$ converge globally to $(M_\infty, g_\infty)$ in the Cheeger-Gromov topology as $t$ tends to $\infty$. If $M$ admits a shrinking K\"ahler-Ricci soliton, then by the uniqueness theorem of Berndtsson \cite{Be13} and Berman-Boucksom-Essydieux-Guedj-Zeriahi \cite{BB12}, the K\"ahler-Ricci flow should converge to the Ricci soliton. This will be discussed in a future paper.
\end{rema}

The proof relies on Perelman's pseudolocality theorem \cite{Pe02} of Ricci flow and a regularity theory for manifolds with integral bounded Ricci curvature. The latter is a generalization of the regularity theory of Cheeger-Colding \cite{ChCo97, ChCo00} and Cheeger-Colding-Tian \cite{ChCoTi02} for manifolds with bounded Ricci curvature. We remark that the uniform noncollapsing condition (\ref{volume noncollaping:0}) also plays a role in the regularity theory; see Section 2 for further discussions.

The central issue is to check the integral condition of Ricci curvature under the K\"ahler-Ricci flow. Indeed we can prove the following partial integral estimate:

\begin{theo}
Let $(M,g(t))$ be as above. There exists some constant $\Lambda$ depending on $g_0$ such that
\begin{equation}\label{Ricci:L4}
\int_M|Ric(g(t))|^4dv_{g(t)}\,\leq\,\Lambda.
\end{equation}
\end{theo}

Therefore, by the regularity result, we have

\begin{coro}
Conjecture \ref{conj:HT} holds for dimension $n\leq3$.
\end{coro}

Inspired by \cite{DoSu12} as well as \cite{Ti12, Ti13}, as a
consequence of Theorem \ref{regularity:1}, we establish the partial $C^0$ estimate for the K\"ahler-Ricci flow (See Section 5 for details).
As a direct corollary, we refine the regularity in Theorem \ref{regularity:1}.

\begin{theo}\label{regularity:2}
Suppose $(M,g(t_i))\stackrel{d_{GH}}{\longrightarrow}(M_\infty,g_\infty)$ as phrased in Theorem \ref{regularity:1}. Then
$M_\infty$ is a normal projective variety and ${\mathcal S}$ is a subvariety of complex codimension at least $2$.
\end{theo}

\begin{rema}
If we consider a K\"ahler-Ricci flow on a normal Fano orbifold, then the limit $M_\infty$ is also a normal variety. The main ingredients in the proof of our regularity of K\"ahler-Ricci flow remain hold for orbifolds. Actually, using the convexity of the regular set one can generalize the regularity theory of Cheeger-Colding and Cheeger-Colding-Tian to orbifolds with integral bounded Ricci curvature. Moreover, Perelman's estimates to Ricci potentials and local volume noncollapsings as well as the pseudolocality theorem keep valid for orbifold K\"ahler-Ricci flow.
\end{rema}

The partial $C^0$ estimate of K\"ahler-Einstein manifolds plays the key role in Tian's program to resolve the Yau-Tian-Donaldson conjecture, see \cite{Ti90}, \cite{Ti97}, \cite{Ti10}, \cite{DoSu12} and \cite{Ti12} for examples. An extension of the partial $C^0$ estimate to shrinking K\"ahler-Ricci solitons was given in \cite{PSS12}. These works are based on the compactness of Cheeger-Colding-Tian \cite{ChCoTi02} and its generalizations to K\"ahler-Ricci solitons by \cite{TiZh12}.
Besides these known cases, the partial $C^0$ estimate conjecture proposed in \cite{Ti90, Ti90b} is still open in general.

Finally we show the Yau-Tian-Donaldson conjecture from the Hamilton-Tian conjecture by the method of K\"ahler-Ricci flow. As discussed before, the key is the partial $C^0$ estimate. One can follow the arguments in \cite{Ti10} and \cite{Ti12}. Let $M$ be K-stable as defined in \cite{Ti97}. Suppose $(M,g(t_i))$, $t_i\rightarrow\infty$, converges in the Cheeger-Gromov topology to a shrinking K\"ahler-Ricci soliton $(M_\infty,g_\infty)$ (maybe with singularities) as in Theorem \ref{regularity:1}. We are going to show that $M_\infty$ is isomorphic to $M$ and $g_\infty$ is Einstein in Section 6, that is, we have
\begin{theo}
Suppose that $M$ is K-stable. If $(M,g(t_i))\stackrel{d_{GH}}{\longrightarrow}(M_\infty,g_\infty)$ as phrased in Theorem \ref{regularity:1}, then $M_\infty$ coincides with $M$ and $g_\infty$ is a K\"ahler-Einstein metric.
\end{theo}


\begin{coro}
The Yau-Tian-Donaldson conjecture holds for dimension $\leq 3$.
\end{coro}



\section{Manifolds with integral bounded Ricci curvature}

In this section, following lines of Cheeger-Colding \cite{ChCo96, ChCo97, ChCo00}, Cheeger-Colding-Tian \cite{ChCoTi02} and Colding-Naber \cite{CoNa12}, we
develop a regularity theory for manifolds with integral bounded Ricci curvature. Let $(M,g)$ be an $m$-dimensional Riemannian manifold satisfying
\begin{equation}\label{Ricci:Lp1}
\int_M|Ric_-|^pdv\leq\Lambda
\end{equation}
for some constants $\Lambda<\infty$ and $p>\frac{m}{2}$, where $Ric_-=\max_{|v|=1}\big(0,-Ric(v,v)\big)$. We may assume $\Lambda\ge 1$ in generality.
For applications to the regularity theory
of K\"ahler-Ricci flow, we shall focus on the case when the manifold $(M,g)$ are uniformly locally noncollapsing in the sense that
\begin{equation}\label{noncollapsing}
\vol(B(x,r))\geq \kappa r^m,\hspace{0.3cm}\forall x\in M,r\leq 1,
\end{equation}
where $\kappa>0$ is a fixed constant. It is remarkable that different phenomena would happen if we replace the condition (\ref{noncollapsing}) by noncollapsing in a definite scale such as $\vol(B(p,1))\geq\kappa$. Actually, due to an example of Yang \cite{Ya92}, for any $p>0$, there exists Gromov-Hausdorff limit space of $m$-manifolds with uniformly $L^p$ bounded Riemannian curvature and $\vol(B(x,1))\geq\kappa$ for any $x$ whose tangent cone at some points may collapse.

The geometry of manifolds with integral bounded Ricci curvature has been studied extensively by Dai, Petersen, Wei et al., see \cite{PeWe00} and references therein. It is also pointed out in \cite{PeWe00} that there should exist a Cheeger-Gromov convergence theory for such manifolds. The critical assumption added here is (\ref{noncollapsing}). The regularity theory without this uniform noncollapsing condition is much more subtle and needs further study.

We start by reviewing some known results for manifolds satisfying (\ref{Ricci:Lp1}) which are proved in \cite{PeWe97, PeWe00}. Together with the segment inequalities proved in Subsection 2.4, these estimates will be sufficient to give a direct generalization of the regularity theory of Cheeger-Colding \cite{ChCo96, ChCo97, ChCo00} and Cheeger-Colding-Tian \cite{ChCoTi02} under noncollapsing condition (\ref{noncollapsing}); cf. \cite{PeWe00}. Then we derive some analytical results including the short-time heat kernel estimate on manifolds with additional assumption (\ref{noncollapsing}) and apply these to derive the Hessian estimate to the parabolic approximations of distance functions as in \cite{CoNa12}. This makes it possible to give a generalization of Colding-Naber's work on the H\"older continuity of tangent cones \cite{CoNa12} on the limit spaces of manifolds satisfying (\ref{Ricci:Lp1}) and (\ref{noncollapsing}).

For simplicity we will denote by $C(a_1,a_2,\cdots)$ a positive constant which depends on the variables $a_1,a_2,\cdots$ but may be variant in different situations.

\subsection{Preliminary results}

For any $x\in M$, let $(t,\theta)\in\mathbb{R}^+\times S_x^{m-1}$ be the polar coordinate at $x$ where $S_x^{m-1}$ is the unit sphere bundle restricted at $x$. Write the Riemannian volume form in this coordinate as
\begin{equation}\label{volume form}
dv\,=\,{\cal A}(t,\theta)dt\wedge d\theta.
\end{equation}
Let $r(y)=d(x,y)$ denote the distance function to $x$. Then an immediate computation in the polar coordinate shows
\begin{equation}\label{Laplace comparison:1}
\triangle r\,=\,\frac{\partial}{\partial r}\log{\cal A}(r,\cdot).
\end{equation}
As in \cite{PeWe97}, introduce the error function of the Laplacian comparison of distances
\begin{equation}\label{Laplace comparison:2}
\psi(r,\theta)\,=\,\bigg(\triangle r(\exp_x(r\theta))-\frac{m-1}{r}\bigg)_+,
\end{equation}
where $a_+=\max(a,0)$. Notice that $\psi$ depends on the base point $x$. For any subset $\Gamma\subset S_x$ define
$$B_\Gamma(x,r)\,=\,\{\exp_x(t\theta)|0\leq t<r,\theta\in\Gamma\}.$$
The following estimate which is proved in \cite{PeWe97} is fundamental in the theory of integral bounded Ricci curvature
\begin{equation}\label{Laplace comparison:3}
\int_{B_\Gamma(x,r)}\psi^{2p}dv\,\leq\, C(m,p)\int_{B_\Gamma(x,r)}|Ric_-|^pdv,\,\forall r>0,p>\frac{m}{2},
\end{equation}
where $C(m,p)=\big(\frac{(m-1)(2p-1)}{2p-m}\big)^p$. Based on this integral estimate, Petersen-Wei proved the following relative volume comparison theorem:

\begin{theo}[\cite{PeWe97}]
For any $p>\frac{m}{2}$ there exists $C(m,p)$ such that the following holds
\begin{equation}\label{volume comparison:1}
\frac{d}{dr}\bigg(\frac{\vol(B_\Gamma(x,r))}{r^{m}}\bigg)^{\frac{1}{2p}}\,\leq\, C(m,p)\bigg(\frac{1}{r^{m}}\int_{B_\Gamma(x,r)}|Ric_-|^pdv\bigg)^{\frac{1}{2p}},\,\forall r>0.
\end{equation}
Integrating gives, for any $r_2>r_1>0$,
\begin{eqnarray}\label{vlume comparison:2}
&&\bigg(\frac{\vol(B_\Gamma(x,r_2))}{r_2^m}\bigg)^{\frac{1}{2p}}-\bigg(\frac{\vol(B_\Gamma(x,r_1))}{r_1^m}\bigg)^{\frac{1}{2p}}\nonumber\\
&&\hspace{3cm}\leq C(m,p)\bigg(r_2^{2p-m}\int_{B_\Gamma(x,r_2)}|Ric_-|^pdv\bigg)^{\frac{1}{2p}}.
\end{eqnarray}
\end{theo}

\begin{rema}
The quantity $r^{2p-m}\int_{B_\Gamma(x,r)}|Ric_-|^pdv$ in above inequality {\rm(\ref{vlume comparison:2})} is scaling invariant. Therefore, under the global integral condition of Ricci curvature {\rm(\ref{Ricci:Lp1})}, the volume ratio $\frac{\vol(B_\Gamma(x,r))}{r^m}$ will become almost monotone whenever the radius $r$ in consideration is sufficiently small. This implies in particular the metric cone structure of the tangent cone on noncollapsing limit spaces.
\end{rema}

\begin{rema}
Under additional assumption {\rm(\ref{noncollapsing})}, the relative volume comparison {\rm(\ref{vlume comparison:2})} gives rise to a volume doubling property of concentric metric balls of small radii \cite{PeWe00}.
\end{rema}

\begin{coro}
Under the assumption {\rm(\ref{Ricci:Lp1})}, the volume has the upper bound
\begin{equation}\label{volume comparison:3}
\vol(B_\Gamma(x,r))\,\leq\,|\Gamma|\cdot r^{m}+C(m,p)\Lambda r^{2p},\hspace{0.5cm}\forall r>0,
\end{equation}
where $|\Gamma|$ denotes the measure of $\Gamma$ as a subset of unit sphere.
\end{coro}

The upper bound of volume of geodesic balls can be refined to the upper bound of areas of geodesic spheres as follows.

\begin{lemm}\label{volume comparison:9}
Under the assumption {\rm(\ref{Ricci:Lp1})}, we have
\begin{equation}\label{volume comparison:7}
\vol(\partial B(x,r))\,\leq\, C(m,p,\Lambda)\cdot r^{m-1},\,\mbox{ when }r\leq 1,
\end{equation}
and
\begin{equation}\label{volume comparison:8}
\vol(\partial B(x,r))\,\leq\, C(m,p,\Lambda)\cdot r^{2p-1},\,\mbox{ when }r\geq 1.
\end{equation}
\end{lemm}
\begin{proof}
When $r\leq 1$, this is exactly the Lemma 3.2 of \cite{DaWe04}. We next use iteration to prove the case $r>1$. For simplicity we only consider $r=2^k$ for $k$ being any positive integers. Other radii bigger than 1 can be attained by a finite step iteration starting from a unique radius between $\frac{1}{2}$ and $1$.

By (\ref{Laplace comparison:1}) and (\ref{Laplace comparison:2}),
\begin{equation}\nonumber
\frac{\partial}{\partial t}\frac{{\cal A}(t,\theta)}{t^{m-1}}\,\leq\,\psi(t,\theta)\frac{{\cal A}(t,\theta)}{t^{m-1}}.
\end{equation}
Integrating over the direction space $S_x^{m-1}$ gives
$$\frac{d}{dt}\frac{\int_{S_x}{\cal A}(t,\theta)d\theta}{t^{m-1}}\,\leq\,\frac{\int_{S_x}\psi(t,\theta){\cal A}(t,\theta)d\theta}{t^{m-1}}.$$
Integrating over an interval of radius $[r,2r]$ gives
\begin{eqnarray}
\frac{\int_{S_x}{\cal A}(2r,\theta)d\theta}{(2r)^{m-1}}-
\frac{\int_{S_x}{\cal A}(r,\theta)d\theta}{r^{m-1}}
&\leq& \int_r^{2r}\frac{\int_{S_x}\psi(t,\theta){\cal A}(t,\theta)d\theta}{t^{m-1}}dt\nonumber\\
&\leq& \frac{1}{r^{m-1}}\int_{B(x,2r)}\psi dv.\nonumber
\end{eqnarray}
By the integral version of mean curvature comparison (\ref{Laplace comparison:3}) and volume comparison (\ref{volume comparison:3}),
$$\int_{B(x,2r)}\psi dv\leq\big(\int_{B(x,2r)}\psi^{2p}\big)^{\frac{1}{2p}}\vol(B(x,2r))^{\frac{2p-1}{2p}}\leq C(m,p,\Lambda)(2r)^{2p-1}.$$
Thus,
$$\frac{\int_{S_x}{\cal A}(2r,\theta)d\theta}{(2r)^{m-1}}\leq
\frac{\int_{S_x}{\cal A}(r,\theta)d\theta}{r^{m-1}}+C(m,p,\Lambda)(2r)^{2p-m}.$$
Put $r_k=2^k$, $k\geq 0$. An iteration then gives
$$\int_{S_x}{\cal A}(r_k,\theta)d\theta\leq C(m,p,\Lambda)r_k^{2p-1},$$
as desired.
\end{proof}

Let $\partial B_\Gamma(x,r)=:\{y=\exp_x(r\theta)|\theta\in\Gamma,\,d(x,y)=r\}.$ By the proof of Lemma 3.2 in \cite{DaWe04} we also have the following volume estimate of $\partial B_\Gamma$ in terms of $|\Gamma|$.

\begin{lemm}
Under the assumption {\rm(\ref{Ricci:Lp1})}, we have
\begin{equation}\label{volume comparison:11}
\vol(\partial B_\Gamma(x,r))\,\leq\,C(m,p,\Lambda)\cdot\big(|\Gamma|r^{m-1}+ r^{2p-1}\big),\,\mbox{ when }r\leq 1.
\end{equation}
\end{lemm}

Next we recall a nice cut-off which is constructed by Petersen-Wei following the idea of Cheeger-Colding \cite{ChCo96}. In the following of this subsection we assume {\rm(\ref{Ricci:Lp1})} and {\rm(\ref{noncollapsing})} hold.

\begin{lemm}[\cite{PeWe00}]\label{cut-off:1}
There exist $r_0=r_0(m,p,\kappa,\Lambda)$ and $C=C(m,p,\kappa,\Lambda)$ such that on any $B(x,r)$, $r\le r_0$, there exists a
cut-off $\phi\in C_0^\infty(B(x,r))$ which satisfies
\begin{equation}
\phi\ge 0,\,\phi\equiv 1\mbox{ in }B(x,\frac{r}{2}),
\end{equation}
and
\begin{equation}
\|\nabla\phi\|_{C^0}^2+\|\triangle\phi\|_{C^0}\leq C r^{-2}.
\end{equation}
\end{lemm}

As in \cite{CoNa12} one can extend the construction to a slightly general case, by using a covering technique based on the volume doubling property. Let $E$ be a closed subset of $M$. Denote the $r$-neighborhood of $E$ by
$$U_r(E)=:\{x\in M| d(x,E)<r\}$$
and let $A_{r_1,r_2}(E)=U_{r_2}\backslash\overline{U}_{r_1}$ be the open annulus of radii $0<r_1<r_2$.

\begin{coro}\label{cut-off:2}
For any $R>0$, there exists $C=C(m,p,\kappa,\Lambda,R)$ such that the following holds. Let $E$ be any closed subset and $0<r_1<10r_2<R$. There exists a
cut-off $\phi\in C^\infty(U_R(E))$ which satisfies
\begin{equation}
\phi\ge 0,\,\phi\equiv 1\mbox{ in }A_{3r_1,\frac{r_2}{3}}(E),\,\phi\equiv 0\mbox{ outside }A_{2r_1,\frac{r_2}{2}}(E),
\end{equation}
and
\begin{equation}
\|\nabla\phi\|_{C^0}^2+\|\triangle\phi\|_{C^0}\leq C r_1^{-2}\,\mbox{ in }A_{2r_1,3r_1}(E),
\end{equation}
and
\begin{equation}
\|\nabla\phi\|_{C^0}^2+\|\triangle\phi\|_{C^0}\leq C r_2^{-2}\,\mbox{ in }A_{\frac{r_2}{3},\frac{r_2}{2}}(E).
\end{equation}
\end{coro}

Finally we recall a bound of Sobolev constants which is essential for Nash-Moser iteration on manifolds with integral bounded Ricci curvature. When the Riemannian manifold is a spatial slice of K\"ahler-Ricci flow on a Fano manifold, the Sobolev constant has a global estimate by \cite{ZhQ} \cite{Ye07}. In the general setting, we have

\begin{lemm}[\cite{Ya92}]\label{Sobolev:1}
Then there exist $r_0=r_0(m,p,\kappa,\Lambda)$ and $C=C(m,p,\kappa,\Lambda)$ such that
\begin{equation}\label{localsobolev:1}
C_s(B(x,r_0))\leq C,\,\forall x\in M.
\end{equation}
\end{lemm}

By a covering technique once again it follows directly that

\begin{coro}\label{Sobolev:2}
For any $R>0$, there exists $C=C(m,p,\kappa,\Lambda,R)$ such that
\begin{equation}\label{localsobolev:2}
C_s(B(x,R))\leq C,\,\forall x\in M.
\end{equation}
\end{coro}

Here, the local Sobolev constant $C_s(B(x,R))$ is defined to be the minimum value of $C_s$ such that
\begin{equation}\label{localsobolev:3}
\bigg(\int f^{\frac{2m}{m-2}}dv\bigg)^{\frac{m-2}{m}}\,\leq\, C_s\int\big(|\nabla f|^2+f^2\big)dv,\,\forall f\in C_0^\infty(B(x,R)).
\end{equation}

\subsection{Heat Kernel estimate}

The aim of this subsection is to prove a heat kernel estimate as well as some geometric inequalities for heat equations on manifolds with integral bounded Ricci curvature.

Let $M$ be a Riemannian manifold satisfying (\ref{Ricci:Lp1}) and (\ref{noncollapsing}) for some constants $p>\frac{m}{2}$, $\Lambda>1$ and $\kappa>0$. We start with the mean value inequality and gradient estimate to heat equations.

Denote by $\oint_A=\frac{1}{\vol(A)}\int_A$ the average integration over the set $A$.

\begin{lemm}
There exists $C=C(m,p,\kappa,\Lambda)$ such that the following holds. For any $0<t_0\leq 1$, and $u(x,t)$, a function in $B(x,\sqrt{t_0})\times[0,t_0]$ satisfying
\begin{equation}
\frac{\partial}{\partial t} u\,=\,\triangle u,
\end{equation}
we have
\begin{equation}\label{mean value:1}
u_+(x,t_0)\,\le\, Ct_0^{-1}\int_{\frac{t_0}{2}}^{t_0}\oint_{B(x,\sqrt{t_0})}u_+,
\end{equation}
\begin{equation}\label{gradient estimate:1}
|\nabla u|^2(x,t_0)\,\le\, Ct_0^{-2}\int_{\frac{t_0}{2}}^{t_0}\oint_{B(x,\sqrt{t_0})}u^2.
\end{equation}
\end{lemm}
\begin{proof}
The estimates follow from the iteration argument of Nash-Moser; see Pages 306-316 of \cite{Chow etc.} for details. The proof of the mean value inequality (\ref{mean value:1}) is standard. We give a proof of (\ref{gradient estimate:1}).

First of all, applying the iteration to the evolution of $|\nabla u|^2$
\begin{equation}
\frac{\partial}{\partial t}|\nabla u|^2=\triangle|\nabla u|^2-2|\nabla\nabla u|^2-2Ric(\nabla u,\nabla u)\leq\triangle|\nabla u|^2+|Ric_-||\nabla u|^2,
\end{equation}
where $|Ric_-|$ is $L^p$ integrable, gives
$$|\nabla u|^2(x,t_0)\leq C(m,p,\kappa,\Lambda)t_0^{-1}\int_{\frac{t_0}{2}}^{t_0}\oint_{B(x,\frac{1}{2}\sqrt{t_0})}|\nabla u|^2.$$
Then we estimate $\int_{\frac{t_0}{2}}^{t_0}\oint_{B(x,\frac{1}{2}\sqrt{t_0})}|\nabla u|^2$ in terms of the $L^2$ norm of $u$ to end up the proof. Write down the evolution equation
$$\frac{\partial}{\partial t}u^2=\triangle u^2-2|\nabla u|^2.$$
Let $\phi\in C_0^\infty(B(x,r)),r=t_0^2,$ be a nonnegative cut off function such that $\phi\equiv 1$ on $B(x,\frac{r}{2})$ and for some $C=C(m,p,\kappa,\Lambda)\ge 2$,
$$|\nabla\phi|^2+|\triangle\phi|\leq Cr^{-2}.$$
See Lemma \ref{cut-off:1}. Multiplying the cut off and integrating on space-time we get,
\begin{eqnarray}
2\int_{\frac{t_0}{2}}^{t_0}\int\phi^2|\nabla u|^2&=&\int_{\frac{t_0}{2}}^{t_0}\int\phi^2\triangle u^2-\int\phi^2 u^2(t_0)+\int\phi^2 u^2(\frac{t_0}{2})\nonumber\\
&\leq&Cr^{-2}\int_{\frac{t_0}{2}}^{t_0}\int_{B(x,r)}u^2+\int\phi^2 u^2(\frac{t_0}{2}).\nonumber
\end{eqnarray}
On the other hand,
\begin{eqnarray}
\frac{d}{dt}\int\phi^2 u^2&=&2\int\phi^2u\triangle u=-2\int\phi^2|\nabla u|^2-4\int\phi u\nabla\phi\nabla u\nonumber\\
&\geq&-2\int|\nabla\phi|^2u^2\geq-Cr^{-2}\int_{B(x,r)}u^2.\nonumber
\end{eqnarray}
We claim that
\begin{equation}
\int\phi^2 u^2(\frac{t_0}{2})\leq 3Cr^{-2}\int_{\frac{t_0}{2}}^{t_0}\int_{B(x,r)}u^2,
\end{equation}
which is sufficient to complete the proof. Actually, if it fails, then for any $t\in[\frac{t_0}{2},t_0]$,
$$\int\phi^2 u^2(t)\geq2Cr^{-2}\int_{\frac{t_0}{2}}^{t_0}\int_{B(r)}u^2,$$
consequently
$$\int_{\frac{t_0}{2}}^{t_0}\int_{B(x,r)}\phi^2u^2\geq\frac{t_0}{2}\cdot2Cr^{-2}\int_{\frac{t_0}{2}}^{t_0}\int_{B(x,r)}u^2
=C\int_{\frac{t_0}{2}}^{t_0}\int_{B(x,r)}u^2,$$
which gives a contradiction..
\end{proof}

\begin{coro}
Assume as in above lemma. If $u$ is harmonic in $B(x,r)$, $r\leq 1$, then
\begin{equation}\label{gradient estimate:2}
|\nabla u|^2(x)\,\le\, Cr^{-2}\oint_{B(x,r)}u^2.
\end{equation}
\end{coro}

\begin{theo}[Heat kernel upper bound]
Let $M$ be a complete Riemannian manifold of dimension $m$ which satisfies {\rm(\ref{Ricci:Lp1})} and {\rm(\ref{noncollapsing})}. Let $H(x,y,t)$ be its heat kernel. There exists positive constant $C=C(m,p,\kappa,\Lambda)$ such that
\begin{equation}
H(x,y,t)\,\le\, Ct^{-\frac{m}{2}}e^{-\frac{d^2(x,y)}{5t}},\,\forall x,y\in M,\,0<t\leq 1.
\end{equation}
\end{theo}
\begin{proof}
There really exists a unique heat kernel on a manifold satisfying (\ref{Ricci:Lp1}) due to a criterion of Grogor'yan \cite{Gr99}. The mean value inequality (\ref{mean value:1}) gives the upper bound of $H$
$$H(x,y,t)\le C(m,p,\kappa,\Lambda)t^{-\frac{m}{2}},\,\forall x\in M,t\leq 1.$$
The Gaussian upper bound $H(x,y,t)$ is concluded from \cite{Gr97}.
\end{proof}

\begin{theo}[Heat kernel lower bound]\label{heat kernel lower bound}
Assume as in above theorem. There exist constants $\tau=\tau(m,p,\kappa,\Lambda)$ and $C=C(m,p,\kappa,\Lambda)$ such that
\begin{equation}
H(x,y,t)\,\ge\, C^{-1}t^{-\frac{m}{2}},
\end{equation}
whenever
\begin{equation}
0<t\leq\tau,\,d(x,y)\leq 10\sqrt{t}.
\end{equation}
\end{theo}
\begin{proof}
We follow the argument of Cheeger-Yau \cite{ChYa81}; see \cite{DaWe04} for a more closer situation. Put $\bar{H}(x,y,t)=(4\pi t)^{-\frac{m}{2}}e^{-\frac{d^2(x,y)}{4t}}$. By DuHamel's principle to the heat equation,
\begin{eqnarray}
H(x,y,t)-\bar{H}(x,y,t)&=&\int_0^t\int\frac{\partial}{\partial s}\bar{H}(x,z,t-s)H(z,y,s)dv(z)ds\nonumber\\
&&+\int_0^t\int\bar{H}(x,z,t-s)\frac{\partial}{\partial s}H(z,y,s)dv(z)ds\nonumber.
\end{eqnarray}
Fix $x\in M$ and let $r(z)=d(x,z)$. An easy calculation shows
$$\frac{\partial}{\partial s}\bar{H}(x,z,t-s)=-\triangle\bar{H}(x,z,t-s)+\frac{r(z)}{2(t-s)}\bigg(\frac{n-1}{r(z)}-\triangle r(z)\bigg)\bar{H}(x,z,t-s).$$
Let $\psi(z)=\big(\triangle r(z)-\frac{n-1}{r(z)}\big)_+$. As shown in \cite{DaWe04}, this implies
\begin{eqnarray}
&&H(x,y,t)-\bar{H}(x,y,t)\nonumber\\
&\ge&-\int_0^t\int\frac{r(z)}{2(t-s)}\psi(z)\bar{H}(x,z,t-s)H(z,y,s)dv(z)ds\nonumber\\
&\ge&-C(m,p,\kappa,\Lambda)\int_0^t\int\psi(z)(t-s)^{-\frac{m+1}{2}}s^{-\frac{m}{2}}e^{-\frac{r^2(z)}{6(t-s)}-\frac{d^2(y,z)}{5s}}dv(z)ds\nonumber,
\end{eqnarray}
where we used in the last inequality that $re^{-\frac{r^2}{5(t-s)}+\frac{r^2}{6(t-s)}}$ has a universal upper bound when $t\leq 1$. By (\ref{Laplace comparison:3}),
\begin{equation}
\int\psi^{2p}\leq C(m,p)\int|Ric_-|^p\leq C(m,p,\Lambda).
\end{equation}
By H\"older inequality,
\begin{eqnarray}
\int\psi(z)e^{-\frac{r^2(z)}{6(t-s)}-\frac{d^2(y,z)}{5s}}dv(z)&\leq&C(m,p,\Lambda)\bigg(\int e^{-\frac{2p}{2p-1}\big(\frac{r^2(z)}{6(t-s)}+\frac{d^2(y,z)}{5s}\big)}dv(z)\bigg)^{\frac{2p-1}{2p}}\nonumber.
\end{eqnarray}
When $0<s\leq\frac{t}{2}$,
$$\int e^{-\frac{2p}{2p-1}\big(\frac{r^2(z)}{6(t-s)}+\frac{d^2(y,z)}{5s}\big)}dv(z)\leq\int e^{-\frac{2p}{2p-1}\frac{d^2(y,z)}{5s}}dv(z)\leq C(m,p,\Lambda)s^{\frac{m}{2}};$$
when $\frac{t}{2}\leq s\leq t$, similarly,
$$\int e^{-\frac{p}{p-1}\big(\frac{r^2(z)}{6(t-s)}+\frac{d^2(y,z)}{5s}\big)}dv(z)\leq C(m,p,\Lambda)(t-s)^{\frac{m}{2}}.$$
Here, in order to derive the explicit upper bound of the integration, we need the upper bound of volume growth of the geodesic spheres centered at $x$ and $y$, namely the estimate (\ref{volume comparison:7}) and (\ref{volume comparison:8}) in Corollary \ref{volume comparison:9}.

Summing up the estimates we obtain
\begin{eqnarray}
&&H(x,y,t)-\bar{H}(x,y,t)\nonumber\\
&\geq&-C(m,p,\kappa,\Lambda)\bigg(\int_0^{\frac{t}{2}}(t-s)^{-\frac{m+1}{2}}s^{-\frac{m}{2}\frac{1}{2p}}ds
+\int_{\frac{t}{2}}^t(t-s)^{-\frac{m+1}{2}+\frac{m}{2}\frac{2p-1}{2p}}s^{-\frac{m}{2}}ds\bigg)\nonumber\\
&=&-C(m,p,\kappa,\Lambda)t^{-\frac{m}{2}+\frac{2p-m}{4p}}.\nonumber
\end{eqnarray}
This is sufficient to get the lower bound of $H(x,y,t)$ when $t$ is small and $d(x,y)\leq 10\sqrt{t}$.
\end{proof}

\begin{coro}\label{sub mean value:1}
Assume as in above theorem. There exist constants $\tau=\tau(m,p,\kappa,\Lambda)$ and $C=C(m,p,q,\kappa,\Lambda)$ such that the following holds. Let $f$ be a nonnegative function satisfying
\begin{equation}\label{super parabolic}
\frac{\partial}{\partial t}f\ge \triangle f-\xi
\end{equation}
where $\xi$ is a space-time function such that $\xi_+$ is $L^q$ integrable for some $q>\frac{m}{2}$ at any time slice $t$. Then
\begin{equation}\label{mean value:3}
\oint_{B(x,r)}f(\cdot,0)dv\leq C\big(f(x,r^2)+r^{2-\frac{m}{q}}\cdot\sup_{t\in[0,r^2]}\|\xi_+(t)\|_q\big),\, \forall x\in M, r\leq\sqrt{\tau}.
\end{equation}
\end{coro}
\begin{proof}
The idea follows \cite{CoNa12}. A direct calculation shows
\begin{eqnarray}
\frac{d}{dt}\int f(y,t)H(x,y,r^2-t)dv(y)&=&\int H(x,y,r^2-t)(\frac{\partial}{\partial t}-\triangle)f(y,t)dv(y)\nonumber\\
&\ge&-\int H(x,y,r^2-t)\xi_+(y,t)dv(y).\nonumber
\end{eqnarray}
Then, by the upper bound of $H$,
\begin{eqnarray}
\int H(x,y,r^2-t)\xi_+(y,t)dv(y)&\le& C(r^2-t)^{-\frac{m}{2}}\int\xi_+(y,t)e^{-\frac{d^2(x,y)}{5(r^2-t)}}dv(y)\nonumber\\
&\le&C(r^2-t)^{-\frac{m}{2}}\|\xi_+(t)\|_q\bigg(\int e^{-\frac{q}{q-1}\frac{d^2(x,y)}{5(r^2-t)}}dv(y)\bigg)^{1-\frac{1}{q}}\nonumber\\
&\le&C(r^2-t)^{-\frac{m}{2q}}\|\xi_+(t)\|_q.\nonumber
\end{eqnarray}
Integrating from $0$ to $r^2$ and applying the lower bound of $H$ and upper bound of $\vol(B(x,r))$, we have
\begin{eqnarray}
f(x,r^2)&\ge&\int f(y,0)H(x,y,r^2)dv(y)-\int_0^{r^2}C(r^2-t)^{-\frac{m}{2q}}\|\xi_+(t)\|_qdt\nonumber\\
&\ge& C^{-1}\oint_{B(x,r)} f(y,0)dv(y)-Cr^{2(1-\frac{m}{2q})}\sup_{t\in[0,r^2]}\|\xi_+(t)\|_q.\nonumber
\end{eqnarray}
The required estimate now follows directly.
\end{proof}

\begin{coro}\label{sub mean value:2}
Assume as above. There exist constants $\tau=\tau(m,p,\kappa,\Lambda)$ and $C=C(m,p,q,\kappa,\Lambda)$ such that the following holds. Let $f$ be a nonnegative function on $M$ satisfying
\begin{equation}\label{super harmonic}
\triangle f\leq\xi
\end{equation}
where $\xi\in L^q$ for some $q>\frac{m}{2}$. Then
\begin{equation}\label{mean value:4}
\oint_{B(x,r)}fdv\leq C\big(f(x)+r^{2-\frac{m}{q}}\cdot\|\xi\|_q\big),\, \forall x\in M, r\leq\sqrt{\tau}.
\end{equation}
\end{coro}

The crucial application is when $f$ is the distance function $d$, in which case we have
$$\triangle d\leq\frac{n-1}{d}+\psi,$$
where $\psi$ has a uniform $L^{2p}$ bound in terms of $\int|Ric_-|^pdv$ by {\rm(\ref{Laplace comparison:3})}.

\begin{rema}
There exists an estimate of same type as in Corollary \ref{sub mean value:1} even if $\|\xi(t)\|_q$ is not bounded but satisfies certain growth condition as $t$ approach 0, for example  $\|\xi(t)\|_q\leq Ct^{-1+\epsilon}$ for some $\epsilon>0$. See Lemma 2.23 for an application.
\end{rema}

\begin{rema}
Trivial examples show that the order of $r$, namely $(2-\frac{m}{q})$, in the estimates {\rm(\ref{mean value:3})} and {\rm(\ref{mean value:4})} is sharp. It infers that the estimates for the parabolic approximations in the next subsection are best.
\end{rema}


\subsection{Parabolic approximations}

Let $M$ be a complete Riemannian manifold of dimension $m$ which satisfies {\rm(\ref{Ricci:Lp1})} and {\rm(\ref{noncollapsing})} for some $\kappa>0$, $p>\frac{m}{2}$ and $\Lambda\ge 1$.

Let us represent some notations we shall use in this subsection. Let $\tau=\tau(m,p,\kappa,\Lambda)$ denote the constant in Corollary \ref{sub mean value:2} and $\delta<\tau$ be fixed small positive constant. In the following of this subsection $C=C(m,p,\kappa,\Lambda,\delta)$ will always be a positive constant depending on the parameters $m,p,\kappa,\Lambda,\delta$.

Pick two base points $p^\pm\in M$ with $d=d(p^+,p^-)\leq\frac{1}{20}\sqrt{\tau}$. Define the annulus
$$A_{r,s}=A_{rd,sd}(\{p,q\}),~~~0<r<s<20.$$
Define functions
\begin{equation}
b^+(x)\,=:\,d(p^+,x)-d(p^+,p^-),\,b^-(x)=:d(p^-,x),
\end{equation}
and
\begin{equation}
e(x)\,=:\,d(p^+,x)+d(p^-,x)-d(p^+,p^-)
\end{equation}
on $M$. The last function $e$ is known as the excess function. Let $\phi$ be a nonnegative cut-off in Corollary \ref{cut-off:2} with $E=\{p^\pm\}$ such that
$$\phi=1\,\mbox{ in }A_{\frac{\delta}{4},8};\,\phi=0\,\mbox{ outside }A_{\frac{\delta}{16},16}$$
and
$$|\nabla\phi|^2+|\triangle\phi|\leq C.$$
Define space-time functions $\textbf{b}^\pm_t$ and $\textbf{e}_t$ by
$$\textbf{b}^\pm_t(x)=\int H(x,y,t)\phi(y)b^\pm(y)dv(y)$$
and
$$\textbf{e}_t(x)=\int H(x,y,t)\phi(y)e(y)dv(y).$$
They are heat solutions with initial $\phi b^\pm$ and $\phi e$ respectively. It is obvious $${\bf e}_t\equiv {\bf b}_t^++{\bf b}_t^-.$$
The aim is to investigate the approximating properties of $\textbf{b}^\pm_t$ to the distance functions $b^\pm$ on the annulus domain $A_{\frac{\delta}{4},8}$. The argument goes as in \cite{CoNa12} without essential difficulties.

We start by noticing that
$$\triangle d(p^\pm,x)\leq\frac{n-1}{d(p^\pm,x)}+\psi^\pm$$
where $\psi^\pm=\big(\triangle d(p^\pm,x)-\frac{n-1}{d(p^\pm,x)}\big)_+$ is the error term of Laplacian comparison of distance functions. Then Corollary \ref{sub mean value:2} gives an immediate corollary.

\begin{coro}
For any $0<\epsilon<\frac{1}{100}\delta$, we have
\begin{equation}\label{parabolic approximate:1}
\oint_{B(x,\epsilon d)}edv\,\leq\, C\big( e(x)+\epsilon^{2-\frac{m}{2p}}d\big),\,\forall x\in A_{\frac{\delta}{4},16}.
\end{equation}
In particular, if $e(x)\le \epsilon^{2-\frac{m}{2p}}d$, then
\begin{equation}
e(y)\,\le\, C\epsilon^{1+\alpha}d,\,\forall y\in B(x,\frac{1}{2}\epsilon d)
\end{equation}
where $\alpha=\frac{1}{m+1}\big(1-\frac{m}{2p})>0$.
\end{coro}

A similar argument as in the proof of Corollary \ref{sub mean value:1} also gives

\begin{lemm}
The followings hold
\begin{equation}\label{parabolic approximate:2}
\triangle {\bf b}^+_t,\,\triangle {\bf b}^-_t,\,\triangle {\bf e}_t\,\leq\,C\big( d^{-1}+ t^{-\frac{m}{4p}}\big),\,\forall 0<t\leq\tau.
\end{equation}
\end{lemm}
\begin{proof}
By direct computation,
$$\triangle(\phi b^+)=\triangle\phi b^++2\langle\nabla\phi,\nabla b^+\rangle+\phi\triangle b^+\leq C(m,p,\kappa,\Lambda,\delta)d^{-1}+\psi^+.$$
Thus,
\begin{eqnarray}
\triangle \textbf{b}^+_t(x)&=&\int\triangle_xH(x,y,t)\phi(y)b^+(y)dv(y)=\int\triangle_yH(x,y,t)\phi(y)b^+(y)dv(y)\nonumber\\
&=&\int H(x,y,t)\triangle_y\big(\phi(y)b^+(y)\big)dv(y)\leq Cd^{-1}+\int H(x,y,t)\psi^+(y)dv(y)\nonumber.
\end{eqnarray}
The last term can be estimated by using the upper bound of $H$ when $t\le 1$,
$$\int H(x,y,t)\psi^+(y)dv(y)\leq Ct^{-\frac{m}{2}}\int e^{-\frac{d^2(x,y)}{5t}}\psi^+(y)dv(y)\leq Ct^{-\frac{m}{4p}}\|\psi^+\|_{2p}.$$
The desired upper bound of $\triangle \textbf{b}^+_t$ then follows from (\ref{Laplace comparison:3}). The proofs of the other two estimates are similar.
\end{proof}

\begin{lemm}
For $t\le\tau$ we have
\begin{equation}\label{parabolic approximate:3}
{\bf e}_t(y)\le C(e(x)+d^{-1}t+t^{1-\frac{m}{4p}}),\,\forall y\in B(x,\sqrt{t});
\end{equation}
\begin{equation}\label{parabolic approximate:4}
|\nabla{\bf e}_t|(x)\le Ct^{-\frac{1}{2}}(e(x)+d^{-1}t+t^{1-\frac{m}{4p}}).
\end{equation}
\end{lemm}
\begin{proof}
First of all, when $t$ is small,
$${\bf e}_t(x)= e(x)+\int_0^t\triangle{\bf e}_s(x)ds\leq e(x)+C(d^{-1}t+t^{1-\frac{m}{4p}}).$$
Then by Lemma \ref{sub mean value:1},
$$\oint_{B(x,3\sqrt{t})}{\bf e}_t\leq C{\bf e}_{7t}(x)\leq C(e(x)+d^{-1}t+t^{1-\frac{m}{4p}}).$$
The mean value inequality shows that for all $y\in B(x,\sqrt{t}),$
\begin{eqnarray}
{\bf e}_t(y)&\leq& Ct^{-1}\int_{\frac{t}{2}}^t\oint_{B(y,\sqrt{t})}{\bf e}_sdvds\nonumber\\
&\leq& Ct^{-1}\int_{\frac{t}{2}}^t\oint_{B(y,2\sqrt{s})}{\bf e}_sdvds\nonumber\\
&\leq& C(e(x)+d^{-1}t+t^{1-\frac{m}{4p}}),\nonumber
\end{eqnarray}
where we also used the volume doubling property. The second estimate is a consequence of the mean value inequality.
\end{proof}

We also have the following lemma as in \cite{CoNa12}:

\begin{lemm}
For $x\in A_{\frac{\delta}{2},4}$ and $t\le\frac{1}{100}\delta^2$ the followings hold
\begin{equation}\label{parabolic approximate:5}
|{\bf b}_t^\pm(x)-b^\pm(x)|\,\le\,C( e(x)+d^{-1}t+t^{1-\frac{m}{4p}});
\end{equation}
\begin{equation}\label{parabolic approximate:6}
|\nabla{\bf b}^\pm_t(x)|^2\,\leq\, 1+Ct^{1-\frac{m}{2p}};
\end{equation}
\begin{equation}\label{parabolic approximate:7}
\oint_{B(x,\sqrt{t})}\big||\nabla{\bf b}^\pm_t|^2-1\big|\,\le\,C(e(x)t^{-\frac{1}{2}}+d^{-1}t^{\frac{1}{2}}+t^{\frac{1}{2}-\frac{m}{4p}});
\end{equation}
\begin{equation}\label{parabolic approximate:8}
\int_{\frac{t}{2}}^t\oint_{B(x,\sqrt{t})}|\Hess{\bf b}^\pm_t|^2\,\le\,C(e(x)t^{-\frac{1}{2}}+d^{-1}t^{\frac{1}{2}}+t^{\frac{1}{2}-\frac{m}{4p}}).
\end{equation}
\end{lemm}
\begin{proof}
We prove the first two estimates; the last two integral estimates can be proved totally same as in \cite{CoNa12}. That (\ref{parabolic approximate:5}) can be derived from the upper bound of ${\bf e}_t={\bf b}_t^++{\bf b}_t^-$ in (\ref{parabolic approximate:3}) and the estimate
$${\bf b}_t^\pm(x)-b^\pm(x)\,\le\, C(d^{-1}t+t^{1-\frac{m}{4p}})$$
which can be proved as in above lemma. To show the upper bound of $|\nabla{\bf b}^\pm_t|$, we first apply the gradient estimate (\ref{gradient estimate:1}) together with the $C^0$ bound of ${\bf b}^\pm_t$ to get
$$|\nabla{\bf b}^\pm_t(x)|\leq Ct^{-\frac{1}{2}}.$$
Then we apply the same trick as in the proof of Corollary \ref{sub mean value:1} to the formula
$$\frac{\partial}{\partial t}|\nabla{\bf b}^\pm_t|\,\le\, \triangle|\nabla{\bf b}^\pm_t|+|Ric_-||\nabla{\bf b}^\pm_t|$$ \
to get
\begin{equation}
|\nabla{\bf b}^\pm_t(x)|\,\le\,1+C\int_0^t\int H(x,y,t-s)s^{-\frac{1}{2}}|Ric_-|dvds\,\le\,1+Ct^{\frac{1}{2}-\frac{m}{2p}}.\nonumber
\end{equation}
Finally repeating this argument we get the desired estimate (\ref{parabolic approximate:6}).
\end{proof}

\begin{rema}
Notice that the order $\frac{1}{2}-\frac{m}{4p}$ of $t$ in {\rm(\ref{parabolic approximate:7})} and {\rm(\ref{parabolic approximate:8})} is sharp. This will play an important role in the proof of H\"older continuity of tangent cones of noncollapsing limit space of manifolds with integral bounded Ricci curvature.
\end{rema}

\begin{rema}
The estimates up to now combining with the segment inequality in next subsection are sufficient to generalize the regularization theory of Cheeger-Colding \cite{ChCo96, ChCo97, ChCo00}, Cheeger-Colding-Tian \cite{ChCoTi02} and Cheeger \cite{Ch03} to manifolds satisfying {\rm(\ref{Ricci:Lp1})} and {\rm(\ref{noncollapsing})}. The point is to get appropriate functions ${\bf b}_t^\pm$ of distance functions $b^\pm$ with small $L^2$ Hessian. We refer to \cite{PeWe00} for related discussions and results.
\end{rema}

Let $\sigma$ be any $\epsilon$-geodesic connecting $p^\pm$ whose length is less than $(1+\epsilon^2)d$. Obviously $e(x)\le\epsilon^2d$ for all $x\in\sigma$. As in \cite{CoNa12} we have a better $L^2$ estimate for $\Hess{\bf b}^\pm_t$ along $\epsilon$-geodesics. Since the proof is same it will be omitted.

\begin{theo}
The following estimates hold for any $\delta d\le t_0<t_0+\sqrt{t}\le(1-\delta)d$,
\begin{equation}\label{parabolic approximate:11}
\int_{t_0}^{t_0+\sqrt{t}}\oint_{B(\sigma(s),\sqrt{t})}\big||\nabla{\bf b}^\pm_{t}|^2-1\big|dvds\,\le\, C\big(\epsilon^2d+d^{-1}t+t^{1-\frac{m}{4p}}\big);
\end{equation}
\begin{equation}\label{parabolic approximate:12}
\int_{\frac{t}{2}}^t\int_{t_0}^{t_0+\sqrt{t}}\oint_{B(\sigma(s),\sqrt{t})}|\Hess{\bf b}^\pm_{\tau}|^2dvdsd\tau \,\le\, C\big(\epsilon^2d+d^{-1}t+t^{1-\frac{m}{4p}}\big).
\end{equation}
\end{theo}

Taking $t=d_\epsilon^2=(\epsilon d)^2$ we obtain

\begin{coro}
There is $r\in[\frac{1}{2},1]$ such that the following estimates hold for any $\delta d\le t_0<t_0+d_\epsilon \le(1-\delta)d$,
\begin{equation}\label{parabolic approximate:13}
\int_{t_0}^{t_0+rd_\epsilon}\oint_{B(\sigma(s),rd_\epsilon)}\big||\nabla{\bf b}^\pm_{r^2d_\epsilon^2}|^2-1\big|dvds\,\le\, C(\epsilon^2d+d_\epsilon^{2-\frac{m}{2p}});
\end{equation}
\begin{equation}\label{parabolic approximate:14}
\int_{t_0}^{t_0+rd_\epsilon}\oint_{B(\sigma(s),rd_\epsilon)}|\Hess{\bf b}^\pm_{r^2d_\epsilon^2}|^2dvds\,\le\, C(d^{-1}+d_\epsilon^{-\frac{m}{2p}}).
\end{equation}
\end{coro}

The following lemma is also needed

\begin{lemm}
The following holds for any $\delta d\le t^{'}<t\le(1-\delta)d$ with $t-t^{'}\le d_\epsilon$,
\begin{equation}\label{parabolic approximate:15}
\int_{t^{'}}^{t}|\nabla{\bf b}_{r^2d_\epsilon^2}^\pm-\nabla b^\pm|(\sigma(s))ds\,\le\, C(\epsilon d^{\frac{1}{2}}+d_\epsilon^{1-\frac{m}{4p}})\sqrt{t-t^{'}}.
\end{equation}
\end{lemm}
\begin{proof}
Notice that
$$\int_{t^{'}}^{t}|\nabla{\bf b}_{r^2d_\epsilon^2}^\pm-\nabla b^\pm|^2(\sigma(s))ds
\,=\,\int_{t^{'}}^{t}\big(|\nabla{\bf b}_{r^2d_\epsilon^2}^\pm|^2-1)+2(1-\langle\nabla{\bf b}_{r^2d_\epsilon^2}^\pm,\nabla b^\pm\rangle\big)ds.$$
A direct calculation as in \cite{CoNa12} gives
$$\int_{t^{'}}^{t^{'}+d_\epsilon}|\nabla{\bf b}_{r^2d_\epsilon^2}^\pm-\nabla b^\pm|^2(\sigma(s))ds
\,\le\,C(\epsilon^2d+d_\epsilon^{2-\frac{m}{2p}}).$$
Then use the Cauchy-Schwartz inequality to get the required estimate.
\end{proof}

\subsection{Segment inequalities}

The segment inequality of Cheeger-Colding \cite{ChCo96} plays an important role in their proof of the local almost rigidity structure \cite{ChCo96, ChCo97} of manifolds with Ricci curvature bounded below. However the proof of the segment inequality depends highly on the pointwise comparison of mean curvature along a radial geodesic \cite{ChCo96} which does not remain valid on manifolds with integral bounded Ricci curvature. We will prove two modified versions which are sufficient for our applications. The first one applies to scalar functions; the second one is a substitution of segment inequality for Hessian estimate.

For any fixed $x\in M$ let ${\cal A}$ and $\psi$ be the volume element and error function in the polar coordinate which are defined by (\ref{volume form}) and (\ref{Laplace comparison:2}).

\begin{lemm}
For any $0<\frac{r}{2}\leq t\leq r$ we have
\begin{equation}\label{volume comparison:5}
{\cal A}(r,\theta)\leq 2^{m-1}{\cal A}(t,\theta)+2^{m-1}\int_t^r\psi(\tau,\theta){\cal A}(\tau,\theta)d\tau.
\end{equation}
\end{lemm}
\begin{proof}
By (\ref{Laplace comparison:3}),
\begin{equation}\label{volume comparison:6}
\frac{\partial}{\partial t}\frac{{\cal A}(t,\theta)}{t^{m-1}}\leq\psi(t,\theta)\frac{{\cal A}(t,\theta)}{t^{m-1}}.
\end{equation}
Integrating gives, for any $0<t<r$,
\begin{eqnarray}\nonumber
\frac{{\cal A}(r,\theta)}{r^{n-1}}&\leq&\frac{{\cal A}(t,\theta)}{t^{n-1}}+\int_t^r\psi(\tau,\theta)\frac{{\cal A}(\tau,\theta)}{\tau^{n-1}}d\tau\nonumber\\
&\leq&\frac{{\cal A}(t,\theta)}{t^{m-1}}+\frac{1}{t^{m-1}}\int_t^r\psi(\tau,\theta){\cal A}(\tau,\theta)d\tau\nonumber.
\end{eqnarray}
The desired estimate now follows immediately.
\end{proof}

For nonnegative function $f$ define $\mathcal{F}_f:M\times M\rightarrow \mathbb{R}^+,$
$$\mathcal{F}_f(x,y)=\inf\bigg\{\int_\gamma f(t)dt\big|\gamma\mbox{ is a minimal normal geodesic from }x\mbox{ to }y.\bigg\}$$

\begin{prop}\label{segment:1}
Let $M$ be a complete Riemannian manifold of dimension $m$ which satisfies {\rm(\ref{Ricci:Lp1})} for some $p>\frac{m}{2}$ and $\Lambda\ge 1$. Then, the following holds for any $B=B(z,R)$, $R\le 1$,
\begin{eqnarray}\label{segment:2}
\int_{B\times B}\mathcal{F}_f(x,y)dv(x)dv(y)&\leq& 2^{m+1} R\vol(B)\int_{B(z,2R)}fdv\\
&+&C(m,p,\Lambda)R^{m+2-\frac{m}{2p}}\vol(B)\|f\|_{C^0(B(z,2R))}.\nonumber
\end{eqnarray}
\end{prop}
\begin{proof}
Denote by $\gamma=\gamma_{x,y}$ a minimal geodesic from $x$ to $y$. Write
$$\mathcal{F}_1(x,y)=\int_0^{\frac{d(x,y)}{2}}f(\gamma(t))dt,\,
\mathcal{F}_2(x,y)=\int_{\frac{d(x,y)}{2}}^{d(x,y)}f(\gamma(t))dt.$$
By symmetry, as in \cite{ChCo96}, it suffices to establish a bound of $\int_{B\times B}\mathcal{F}_2(x,y)dv(x)dv(y)$.

Fix $x\in B$. If $y=\exp_x(r\theta)\in B$, $r=d(x,y)$,
\begin{eqnarray}
{\cal A}(r,\theta)\int_{\frac{r}{2}}^{r}f(\gamma(t))dt&\leq&
2^{m-1}\int_{\frac{r}{2}}^r f(\gamma(t)){\cal A}(t,\theta)dt\nonumber\\
&&+2^{m-1}\int_{\frac{r}{2}}^r
\bigg(\int_t^\rho f(\gamma(t))\psi(\tau,\theta){\cal A}(\tau,\theta)d\tau\bigg)dt\nonumber\\
&\leq&2^{m-1}\int_0^r f(\gamma(t)){\cal A}(t,\theta)dt\nonumber\\
&&+2^{m}R\|f\|_{C^0}\int_0^r\psi(\tau,\theta){\cal A}(\tau,\theta)d\tau\nonumber.
\end{eqnarray}
Integrating over $B$ gives
$$\int_{B}\mathcal{F}_2(x,y)dv(y)\leq2^{m}R\int_{B(z,2R)}fdv+2^{m}R^2\|f\|_{C^0}\int_{B(x,2R)}\psi.$$
By (\ref{Laplace comparison:3}) and volume growth estimate (\ref{volume comparison:3}) we get
$$\int_{B(x,2R)}\psi\leq C(m,p,\Lambda)R^{m(1-\frac{1}{2p})}.$$
This is sufficient to complete the proof.
\end{proof}

For any $x,y\in M$ let $\gamma_{x,y}$ be a minimizing normal geodesic connecting $x$ and $y$.

\begin{prop}\label{segment:3}
Let $f\in C^\infty(B(z,3R))$, $R\le 1$, satisfing $|\nabla f|\,\le\,\Lambda^{'}.$ For any $\eta>0$ the following holds
\begin{eqnarray}\label{segment:4}
&&\int_{B(z,R)\times B(z,R)}\big|\langle\nabla f,\dot{\gamma}_{x,y}\rangle(x)-\langle\nabla f,\dot{\gamma}_{x,y}\rangle(y)\big|dv(x)dv(y)\\
&&\hspace{1cm}\le C(m)\eta^{-1}R^{m+1}\int_{B(z,3R)}\big|\Hess f\big|dv
+C(m,p,\Lambda)\cdot\Lambda^{'}\cdot(R^{2p}+\eta R^m)\vol(B).\nonumber
\end{eqnarray}
\end{prop}
\begin{proof}
Let $B=B(z,R)$. Fix $x\in B$ and view points of $B$ in polar coordinate at $x$. Define for any $\theta\in S_x$ the maximum radius $r(\theta)$ such that $\exp_x(r\theta)\in B$ and $d(x,\exp_x(r\theta))=r$. Obviously $r(\theta)\le 2R$. Let $\gamma_\theta(t)=\exp_x(t\theta)$, $t\le r(\theta)$, be a radial geodesic for $\theta\in S_x$. Then
\begin{eqnarray}
\int_{\{x\}\times B}\big|\langle\nabla f,\dot{\gamma}_{x,y}\rangle(x)-\langle\nabla f,\dot{\gamma}_{x,y}\rangle(y)\big|dv(y)
&\le&\int_{S_x}\int_0^{r(\theta)}\big|\langle\nabla f,\dot{\gamma}_\theta\rangle(0)-\langle\nabla f,\dot{\gamma}_\theta\rangle(t)\big|{\cal A}(t,\theta)dtd\theta\nonumber\\
&\le&\int_0^{2R}\int_{S_x}\big|\langle\nabla f,\dot{\gamma}_\theta\rangle(0)-\langle\nabla f,\dot{\gamma}_\theta\rangle(t)\big|{\cal A}(t,\theta)d\theta dt\nonumber.
\end{eqnarray}
Then we divide the integration into two parts, for each $t\in[0,2r]$,
$$\int_{\{{\cal A}(t,\theta)\le \eta^{-1}t^{m-1}\}}\big|\langle\nabla f,\dot{\gamma}_\theta\rangle(0)-\langle\nabla f,\dot{\gamma}_\theta\rangle(t)\big|{\cal A}(t,\theta)d\theta\,\le\,\eta^{-1}t^{m-1}\int_{S_x}\big|\langle\nabla f,\dot{\gamma}_\theta\rangle(0)-\langle\nabla f,\dot{\gamma}_\theta\rangle(t)\big|d\theta,$$
$$\int_{\{{\cal A}(t,\theta)> \eta^{-1}t^{m-1}\}}\big|\langle\nabla f,\dot{\gamma}_\theta\rangle(0)-\langle\nabla f,\dot{\gamma}_\theta\rangle(t)\big|{\cal A}(t,\theta)d\theta\,\le\,2\Lambda^{'}\int_{\{{\cal A}(t,\theta)>\eta^{-1}t^{m-1}\}}{\cal A}(t,\theta)d\theta.$$
By (\ref{volume comparison:7}),
$$\big|\{{\cal A}(t,\theta)>\eta^{-1}t^{m-1}\}\big|\,\le\, C(m,p,\Lambda)\eta.$$
Then (\ref{volume comparison:11}) gives
$$\int_{\{{\cal A}(t,\theta)>\eta^{-1}t^{m-1}\}}{\cal A}(t,\theta)d\theta\,\le\,C(m,p,\Lambda)(\eta t^{m-1}+t^{2p-1}).$$
Therefore,
\begin{eqnarray}
&&\int_{\{x\}\times B}\big|\langle\nabla f,\dot{\gamma}_{x,y}\rangle(x)-\langle\nabla f,\dot{\gamma}_{x,y}\rangle(y)\big|dv(y)\nonumber\\
&&\le
\eta^{-1}\int_0^{2R}\int_{S_x}\big|\langle\nabla f,\dot{\gamma}_\theta\rangle(0)-\langle\nabla f,\dot{\gamma}_\theta\rangle(t)\big|t^{m-1}d\theta dt
+C(m,p,\Lambda)\cdot\Lambda^{'}\cdot (R^{2p}+\eta R^m)\nonumber\\
&&\le C(m)\eta^{-1}R^{m-1}\int_0^{2R}\int_{S_x}\big|\langle\nabla f,\dot{\gamma}_\theta\rangle(0)-\langle\nabla f,\dot{\gamma}_\theta\rangle(t)\big|d\theta dt
+C(m,p,\Lambda)\cdot\Lambda^{'}\cdot (R^{2p}+\eta R^m)\nonumber\\
&&\le C(m)\eta^{-1}R^m\int_0^{2R}\int_{S_x}\big|\Hess f\big|(\gamma_\theta(t))d\theta dt
+C(m,p,\Lambda)\cdot\Lambda^{'}\cdot (R^{2p}+\eta R^m)\nonumber.
\end{eqnarray}
Integrating over $x\in B$ gives
\begin{eqnarray}
&&\int_{B\times B}\big|\langle\nabla f,\dot{\gamma}_{x,y}\rangle(x)-\langle\nabla f,\dot{\gamma}_{x,y}\rangle(y)\big|dv(x)dv(y)\nonumber\\
&&\le C(m)\eta^{-1}R^m\int_0^{2R}\int_{SB}\big|\Hess f\big|(\gamma_\theta(t))d\theta dt
+C(m,p,\Lambda)\cdot\Lambda^{'}\cdot (R^{2p}+\eta R^m)\vol(B)\nonumber\\
&&\le C(m)\eta^{-1}R^{m+1}\int_{B(z,3R)}\big|\Hess f\big|dv
+C(m,p,\Lambda)\cdot\Lambda^{'}\cdot(R^{2p}+\eta R^m)\vol(B).\nonumber
\end{eqnarray}
The last inequality uses the invariance of Liouville measure under geodesic flow.
\end{proof}

\subsection{Almost rigidity structures}

Let $M$ be a complete Riemannian manifold of dimension $m$ which satisfies {\rm(\ref{Ricci:Lp1})} and {\rm(\ref{noncollapsing})} for some $\kappa>0$, $p>\frac{m}{2}$ and $\Lambda\ge 1$. The local almost rigidity properties below can be proved exactly as in \cite{ChCo96, ChCo97}.

For any $\epsilon>0$ small, there exist positive constants $\delta$, $r_0$ depending on $m,p,\kappa,\Lambda$ and $\epsilon$ such that the following theorems \ref{almost rigidity:1}-\ref{almost rigidity:4} hold.

\begin{theo}[\cite{PeWe00}, Almost splitting]\label{almost rigidity:1}
Let $p^\pm\in M$ with $d=d(p^+,p^-)\le r_0$. If $x\in M$ satisfies $d(p^\pm,x)\ge\frac{1}{5} d$ and
\begin{equation}
d(p^+,x)+d(p^-,x)-d\,\le\,\delta^2 d,
\end{equation}
then there exists a complete length space $X$ and $B((0,x^*),r)\subset\mathbb{R}\times X$ such that
\begin{equation}
d_{GH}\big(B(x,\delta d),B((0,x^*),\delta d)\big)\,\leq\,\epsilon d.
\end{equation}
\end{theo}

\begin{theo}[\cite{PeWe00}, Volume convergence]\label{almost rigidity:2}
If $x\in M$ satisfies
\begin{equation}
d_{GH}\big(B(x,r),B_r\big)\,\leq\,\delta r,
\end{equation}
for some $r\le r_0$, where $B_r$ denotes an Euclidean ball of radius $r$, then
\begin{equation}\label{volumeconcollap2}
\vol(B(x,r))\,\geq\,(1-\epsilon)\vol(B_r).
\end{equation}
\end{theo}

\begin{theo}[Almost metric cone]\label{almost rigidity:3}
If $x\in M$ satisfies
\begin{equation}\label{almost volume cone}
\frac{\vol(B(x,2r))}{\vol(B_{2r})}\,\geq\,(1-\delta)\frac{\vol(B(x,r))}{\vol(B_r)},
\end{equation}
for some $r\le r_0$, then there exists a compact length space $X$ with
\begin{equation}\label{diameter bound}
\diam(X)\,\leq\,(1+\epsilon)\pi
\end{equation}
such that, for metric ball $B(o^*,r)\subset C(X)$ centered at the vertex $o^*$,
\begin{equation}
d_{GH}\big(B(x,r),B(o^*,r)\big)\,\leq\,\epsilon r.
\end{equation}
\end{theo}

\begin{theo}\label{almost rigidity:4}
If $x\in M$ satisfies
\begin{equation}
\vol(B(x,2r))\,\geq\,(1-\delta)\vol(B_{2r})
\end{equation}
for some $r\le r_0$, then
\begin{equation}
d_{GH}\big(B(x,r),B_r\big)\,\leq\,\epsilon r.
\end{equation}
\end{theo}

\subsection{$C^\alpha$ structure in almost Euclidean region}

Let $M$ be a complete Riemannian manifold of dimension $m$ which satisfies {\rm(\ref{noncollapsing})} for some $\kappa>0$. Instead of {\rm(\ref{Ricci:Lp1})} we also assume the following $L^p$ bound of Ricci curvature
\begin{equation}\label{Ricci:Lp3}
\int_M|Ric|^p\,\leq\,\Lambda,
\end{equation}
for some $p>\frac{m}{2}$ and $\Lambda\ge 1$.

Fix $\alpha\in(0,1)$ and $\theta>0$. For $x\in M$, define the \textit{$C^\alpha$ harmonic radius} at $x$, denoted by $r_g^{\alpha,\theta}(x)$, to be the maximal radius $r$ such that there exists a harmonic coordinate $\textrm{X}=(x^1,\cdots,x^{2n}):B(x,r)\rightarrow\mathbb{R}^{2n}$ which satisfies
\begin{equation}\label{harmonic coordinate:1}
e^{-\theta}(\delta_{ij})\leq(g_{ij})\leq e^\theta(\delta_{ij})
\end{equation}
as matrices, and
\begin{equation}\label{harmonic coordinate:2}
\sup_{i,j}\big(\|g_{ij}\|_{C^0}+r^\alpha\|g_{ij}\|_{C^\alpha}\big)\leq e^\theta,
\end{equation}
where $g_{ij}=(\textrm{X}^{-1})^*g(\frac{\partial}{\partial x^i},\frac{\partial}{\partial x^j})$ is defined on the domain $\textrm{X}(B(x,r))$. In harmonic coordinates, the $L^p$ bound of Ricci curvature gives the $L^{2,p}$ bound of the metric tensor $g_{ij}$ which in turn implies the $C^\alpha$ regularity of metric. Following the arguments in \cite{An90} and \cite{Pe96} one can prove

\begin{theo}
For any $\delta,\theta\in (0,1)$ and $0<\alpha<2-\frac{m}{p}$, there exist $\eta>0$ and $r_0>0$ such that the following holds: if $x\in M$ satisfies
\begin{equation}\label{harmonic coordinate:3}
\vol(B(x,r))\,\ge\,(1-\eta)\vol(B_r)
\end{equation}
for some $r\le r_0$, then
\begin{equation}\label{harmonic coordinate:4}
r_g^{\alpha,\theta}(x)\,\geq\,\delta r.
\end{equation}
\end{theo}

\begin{coro}
Assume as in above theorem. If $x\in M$ satisfies {\rm(\ref{harmonic coordinate:3})}, then the isoperimetric constant of $B(x,\delta r)$ has a lower bound
\begin{equation}
\Isop(B(x,\delta r))\geq(1-\theta)\Isop(\mathbb{R}^m).
\end{equation}
\end{coro}

\subsection{Structure of the limit space}

Let $(M_i,g_i)$ be a sequence of Riemannian manifolds of dimension $m$  which satisfies {\rm(\ref{noncollapsing})} and {\rm(\ref{Ricci:Lp3})} for some $\kappa,\Lambda>0$ and $p>\frac{m}{2}$ independent of $i$. Then (\ref{volume comparison:3}) gives us the uniform upper bound of volume growth. By Gromov's first convergence theorem, there exists a complete length metric space $(Y,d)$ such that,
\begin{equation}\label{convergence3}
(M_i,g_i)\stackrel{d_{GH}}{\longrightarrow}(Y,d)
\end{equation}
along a subsequence in the pointed Gromov-Hausdorff topology.

\begin{theo}\label{theorem:Lplimit}
Assume as above, the followings hold,
\begin{itemize}
\item[(i)] for any $r>0$ and $x_i\in M_i$ such that $x_i\rightarrow x_\infty\in Y$, we have
\begin{equation}\label{volumeconvergence}
\vol(B(x_i,r))\rightarrow\mathcal{H}^{m}(B(x_\infty,r)),
\end{equation}
where $\mathcal{H}^{m}$ denotes the $m$-dimensional Hausdorff measure;
\item[(ii)] For any $x\in Y$ and any sequence $\{r_j\}$ with $\lim r_j= 0$, a subsequence of $(Y, r^{-2}_j d, x)$ converges to a metric space $({\cal C}_x, d_x, o)$. Any such a $({\cal C}_x, d_x, o)$ is a metric cone with vertex $o$ and splits off lines isometrically;
\item[(iii)] $Y=\mathcal{S}\cup\mathcal{R}$ such that $\mathcal{S}$ is a closed set of codimension $\geq 2$ and $\mathcal{R}$ is convex in $Y$; $\mathcal{R}$ consists of points whose tangent cone is $\mathbb{R}^m$;
\item[(iv)] There is a $C^{1,\alpha}$-smooth structure on $\mathcal{R}$ and a $C^{\alpha}$, $\forall\alpha<2-\frac{m}{p}$, metric $g_\infty$ there which induces $d$; moreover, $g_i$ converges to $g_\infty$ in the $C^{1,\alpha}$ topology on $\mathcal{R}$;
\item[(v)] The singular set $\mathcal{S}$ has codimension $\geq 4$ if each $(M_i, g_i)$ is K\"ahlerian.
\end{itemize}
\end{theo}

The proofs of (i)-(iv), except the convexity of $\mathcal{R}$, are standard, following the same line as that of Cheeger-Colding and Cheeger-Colding-Tian; see \cite{Ch03, ChCo96, ChCo97, ChCoTi02}. In the K\"ahler setting, the convergence of the metric and complex structure takes place in the $C^\alpha\cap L^{2,p}$ topology on $\mathcal{R}$ (cf. \cite{Pe96}), so $g_\infty$ is K\"ahler with respect to the limit complex structure in the weak sense. However, the $L^{2,p}$ convergence of $g_i$ will be enough to carry out the slice argument as in \cite{ChCoTi02} or \cite{Ch03} to show the codimension 4 property of the singular set $\mathcal{S}$. The convexity of $\mathcal{R}$ is a consequence of the following local H\"older continuity of geodesic balls in the interior of geodesic segments and the local $C^\alpha$ structure of the regular set; see \cite{CoNa12} for details.

\begin{theo}
Let $(M,g)$ be a complete Riemannian manifolds of dimension $m$ which satisfies {\rm(\ref{noncollapsing})} and {\rm(\ref{Ricci:Lp3})} for some $\kappa,\Lambda>0$ and $p>\frac{m}{2}$. There is $\alpha=\alpha(p,m)>0$ such that the following holds. For any $\delta>0$ small, we can find positive constants $C$, $r_0$ depending on $m,p,\kappa,\Lambda,\delta$ such that on any normal geodesic $\gamma:[0,l]\rightarrow M$ of length $l\le 1$,
\begin{equation}\label{Holder continuity}
d_{GH}\big(B_r(\gamma(s)),B_r(\gamma(t))\big)\le\frac{C}{\delta l}|s-t|^\alpha r,
\end{equation}
whenever
$$0<r\le r_0\delta l,\,\delta l\le s\le t\le (1-\delta)l,\, |s-t|\le r.$$
\end{theo}
\begin{proof}
The proof goes totally same as in the proof of Theorem 1.1 in \cite[Section 3]{CoNa12}. As $\alpha\le 1$ and $l\le 1$, we may assume $l=1$ by a scaling. The conditions {\rm(\ref{noncollapsing})} and {\rm(\ref{Ricci:Lp3})} remain valid for same constants $\kappa,\Lambda>0$ and $p>\frac{m}{2}$. To apply the segment inequality established in section 2.4, we replace the Hessian estimate along a geodesic connecting $x$ and $y$, namely $\int_{\gamma_{x,y}}|\Hess {\bf b}^\pm_{r^2}|$ where ${\bf b}^\pm_{r^2}$ is the parabolic approximation of distance function defined in subsection 2.3 with base points $p=\gamma(0)$ and $q=\gamma(l)$, by the integrand in (\ref{segment:4})
$$F(x,y)=:\big|\langle\nabla {\bf b}^\pm_{r^2},\dot{\gamma}_{x,y}\rangle(x)-\langle\nabla {\bf b}^\pm_{r^2},\dot{\gamma}_{x,y}\rangle(y)\big|.$$
Define $I^r_{t-t^{'}}$, $T^r_\eta$ and $T^r_{\eta}(x)$ for $x\in T^r_\eta,\eta>0$ as in \cite{CoNa12}, just by replacing the upper bound of $e_{p,q}$ in $T^r_\eta$ by $\eta^{-1}r^{2-\frac{m}{2p}}$. The points in $T^r_\eta(x)$ behaves very well under the gradient flow associated to the distance function at $p$. In the most simple case, by the Hessian estimate and (\ref{parabolic approximate:15}) in subsection 2.3, if $\gamma(t)\in T^r_\eta$, $t\in[\delta,1-\delta]$, and $x\in T^r_\eta(\gamma(t))\cap T^r_\eta$ as considered in page 1210 of \cite{CoNa12} , the distortion of distance under the geodesic flow can be estimated as follows
\begin{eqnarray}
d(\gamma_{p,x}(t^{'}),\gamma(t^{'}))-d(x,\gamma(t))&\le& C\eta^{-2}\big[\delta^{-1}r^{-\frac{m}{4p}}\sqrt{t-t^{'}}+\delta r^{-1}(t-t^{'})+r^{2p-m-1}(t-t^{'})\big]r\nonumber\\
&\le&C\eta^{-2}\big[\delta^{-1}(t-t^{'})^{\frac{1}{2}-\frac{m}{4p}}+\delta+(t-t^{'})^{2p-m}\big]r\nonumber
\end{eqnarray}
for all $\delta>0$ and $t^{'}\le t\le t^{'}+r$. It follows by picking $\delta=(t-t^{'})^{\frac{1}{4}-\frac{m}{8p}}$,
$$d(\gamma_{p,x}(t^{'}),\gamma(t^{'}))-d(x,\gamma(t))\,\le\,C\eta^{-2}(t-t^{'})^{\frac{1}{4}-\frac{m}{8p}}r,\,\forall \delta\le t^{'}\le t\le t^{'}+r\le 1-\delta.$$
For general case where $\gamma(t)$ is not in $T^r_\eta$, one can follow \cite{CoNa12} to get a precise $\alpha=\alpha(p,m)$ such that (\ref{Holder continuity}) holds for certain constant $C$.
\end{proof}


\section{Regularity under K\"aher-Ricci flow}

In this section, we prove Theorem \ref{regularity:1}. We will first show that the regular set $\mathcal{R}$ of the limit space is smooth, and then apply the Pseudolocality theorem of Perelman \cite{Pe02}.

Let us fix some notions first. Assume $(M;J)$ is a compact K\"ahler manifold of (complex) dimension $n$. Let $g$ and $\nabla^L$ denote a K\"ahler metric and associated Levi-Civita connection. In local complex coordinate $(z^1,\cdots,z^n)$, define $g_{i\bar{j}}=g(\frac{\partial}{\partial z^i},\frac{\partial}{\partial\bar{z}^j})$ and $R_{i\bar{j}}=Ric(\frac{\partial}{\partial z^i},\frac{\partial}{\partial\bar{z}^j})$, etc. Let $\nabla_i$ and $\nabla_{\bar{j}}$ be the abbreviations of $\nabla^L_{\frac{\partial}{\partial z^i}}$ and $\nabla^L_{\frac{\partial}{\partial\bar{z}^j}}$ for simplicity. Define the projections of Levi-Civita connection onto the $(1,0)$ and $(0,1)$ spaces as
$$\nabla\,=\,\nabla_{i}\otimes dz^i,\,\bar{\nabla}\,=\,\nabla_{\bar{i}}\otimes d\bar{z}^i.$$
Define the rough Laplacian acting on tensor fields $\triangle=g^{i\bar{j}}\nabla_i\nabla_{\bar{j}}.$

From now on, $(M;J)$ will be a compact Fano $n$-manifold and $g_0$ is a K\"ahler metric in the anti-canonical class $2\pi c_1(M;J)$. Let $g(t)$ be the solution to the volume normalized K\"ahler-Ricci flow
\begin{equation}
\frac{\partial g}{\partial t}\,=\,g\,-\,{\rm Ric}(g)
\end{equation}
with initial $g(0)=:g_0$. By $\partial\bar{\partial}$-lemma, there exists a family of real-valued functions $u(t)$, called Ricci potentials of $g(t)$, which are determined by
\begin{equation}\label{Ricci potential}
g_{i\bar{j}}-R_{i\bar{j}}\,=\,\partial_i\partial_{\bar{j}}u,\,\frac{1}{\V}\int e^{-u(t)}dv_{g(t)}\,=\,1
\end{equation}
where $\V=\int dv_g$ denotes the volume of the K\"ahler-Ricci flow. By Perelman's estimate (see \cite{SeTi08} for a proof), there exists $C$ depending only on initial metric $g_0$ such that
\begin{equation}\label{perelman bound:Ricci potential1}
\|u(t)\|_{C^0}+\|\nabla u(t)\|_{C^0}+\|\triangle u(t)\|_{C^0}\,\le\, C.
\end{equation}
By Perelman's noncollapsing theorem for Ricci flow \cite{Pe02}, there exist positive constants $\kappa$ and $D$ depending on $g_0$ such that
\begin{equation}\label{perelman bound:noncollapsing}
\vol(B_{g(t)}(x,r))\,\ge\,\kappa r^{2n},\,\forall x\in M,r\leq 1
\end{equation}
\begin{equation}\label{perelman bound:diameter}
\diam(M,g(t))\,\le\, D.
\end{equation}

The following formulas for $u(t)$ can be easily checked under the K\"ahler-Ricci flow
\begin{eqnarray}
\frac{\partial}{\partial t}u&=&\triangle u+u-a;\\
\frac{\partial}{\partial t}|\nabla u|^2&=&\triangle|\nabla u|^2-|\nabla\nabla u|^2-|\nabla\bar{\nabla}u|^2+|\nabla u|^2;\\
\frac{\partial}{\partial t}\triangle u&=&\triangle\triangle u-|\nabla\bar{\nabla} u|^2+\triangle u;\\
\frac{\partial}{\partial t}|\nabla\bar{\nabla}u|^2&=&\triangle|\nabla\bar{\nabla}u|^2-2|\nabla\nabla\bar{\nabla}u|^2
+2R_{i\bar{j}k\bar{l}}\nabla_{\bar{i}}\nabla_lu\nabla_j\nabla_{\bar{k}}u\\
\frac{\partial}{\partial t}|\nabla\triangle u|^2&=&\triangle|\nabla\triangle u|^2-|\nabla\nabla\triangle u|^2-|\nabla\bar{\nabla}\triangle u|^2+|\nabla\triangle u|^2\nonumber\\
&&-\nabla_i|\nabla\bar{\nabla}u|^2\nabla_{\bar{i}}\triangle u-\nabla_i\triangle u\nabla_{\bar{i}}|\nabla\bar{\nabla}u|^2.
\end{eqnarray}
Here,
\begin{equation}\label{average potential}
a(t)\,=\,\frac{1}{\V}\int u(t)e^{-u(t)}dv_{g(t)},
\end{equation}
is the average of $u(t)$. By the Jensen inequality, $a(t)\leq 0$. It is known that $a(t)$ increases along the K\"ahler-Ricci flow \cite{Zh11}, so we may assume
\begin{equation}
\lim_{t\rightarrow\infty} a(t)\,=\,a_\infty.
\end{equation}

\subsection{Long-time behavior of Ricci potentials}

We will show that the Ricci potentials $u(t)$ behaves very well as $t\rightarrow\infty$ under the K\"ahler-Ricci flow, namely its gradient field tends to be holomorphic in the $L^2$ topology. This implies that the limit of K\"ahler-Ricci flow should be K\"ahler-Ricci soliton (in certain weak topology).

\begin{prop}\label{prop1}
Under the K\"ahler-Ricci flow,
\begin{equation}\label{L2bound:Ricci potential1}
\int_0^\infty\int_M|\nabla\nabla u|^2dvdt<\infty.
\end{equation}
In particular,
\begin{equation}\label{L2bound:Ricci potential2}
\int_M|\nabla\nabla u|^2dv\rightarrow 0,\hspace{0.5cm}\mbox{ as }t\rightarrow\infty.
\end{equation}
\end{prop}

\begin{prop}\label{prop2}
Under the K\"ahler-Ricci flow,
\begin{equation}\label{L2bound:Ricci potential3}
\int_t^{t+1}\int_M\big|\nabla(\triangle u-|\nabla u|^2+u)\big|^2dvdt\rightarrow 0,\hspace{0.5cm}\mbox{ as }t\rightarrow\infty,
\end{equation}
\begin{equation}\label{L2bound:Ricci potential4}
\int_M(\triangle u-|\nabla u|^2+u-a)^2dv\rightarrow 0,\hspace{0.5cm}\mbox{ as }t\rightarrow\infty.
\end{equation}
\end{prop}

\begin{rema}
Proposition \ref{prop1} gives a hint of the global convergence of a K\"ahler-Ricci flow. Assuming the boundedness of curvature, this has been proved by Ache \cite{Ac12}.
\end{rema}

\begin{rema}
Recall that a Riemannian manifold $(M,g)$ is a shrinking Ricci soliton if
\begin{equation}
Ric+\Hess f=\lambda g
\end{equation}
for some $f\in C^\infty(M;\mathbb{R})$ and $\lambda>0$. In the case $M$ is Fano and $g\in 2\pi c_1(M)$, the manifold is a shrinking Ricci soliton (always called shrinking K\"ahler-Ricci soliton) only if $\lambda=1$ and $f$ equals the Ricci potential $u$. In other words, $(M,g)$ is a shrinking K\"ahler-Ricci soliton if and only if
\begin{equation}
\nabla\nabla u\,=\,0.
\end{equation}
Moreover, applying the Bianchi identity, it can be checked that $(M,g)$ is a shrinking K\"ahler-Ricci soliton if and only if the Shur type identity holds
\begin{equation}
\triangle u-|\nabla u|^2+u\,=\,a.
\end{equation}
\end{rema}

To prove the proposition, we need Perelman's entropy functional (compare Perelman's original definition in \cite{Pe02}): For any K\"ahler metric $g\in 2\pi c_1(M)$, let
\begin{equation}
\mathcal{W}(g,f)\,=\,\frac{1}{\V}\int_M(s+|\nabla f|^2+f-n)e^{-f}dv\,~~~\forall f\in C^\infty(M;\mathbb{R})
\end{equation}
and define
\begin{equation}
\mu(g)\,=\,\inf\bigg\{\mathcal{W}(g,f)\bigg|\int_Me^{-f}dv=\V\bigg\},
\end{equation}
where $s$ is the scalar curvature of $g$. It is known a smooth minimizer of $\mu$, though may not be unique, always exists \cite{Ro81}. The entropy admits a natural upper bound
$$\mu(g)\,\le\,\frac{1}{\V}\int_Mue^{-u}dv=:a\leq 0.$$

Consider the entropy under the K\"ahler-Ricci flow $g(t)$: for any solution $f(t)$ to the backward heat equation
\begin{equation}\label{backheat}
\frac{\partial f}{\partial t}\,=\,-\triangle f+|\nabla f|^2+\triangle u,
\end{equation}
we have
\begin{equation}
\frac{d}{dt}\mathcal{W}(g,f)\,=\,\frac{1}{\V}\int_M\big(|\nabla\bar{\nabla}(u-f)|^2+|\nabla\nabla f|^2\big)e^{-f}dv.
\end{equation}
This implies Perelman's monotonicity
\begin{equation}
\mu(g_0)\,\le\,\mu(g(t))\,\le\, 0,\,\forall t\geq 0.
\end{equation}

We also need the following lemma to prove the propositions:

\begin{lemm}\label{lemma2}
For any $g=g(t)$ and smooth function $f$ we have
\begin{equation}
\int_M|\nabla\nabla f|^2dv\,\le\, C(g_0)\int_M|\nabla\bar{\nabla}f|^2dv.
\end{equation}
\end{lemm}
\begin{proof}
By adding a constant we may assume $f$ satisfies $\int fe^{-u}dv=0$. Then the weighted Poincar\'e inequality \cite{Fu} gives
\begin{equation}\nonumber
\int f^2e^{-u}dv\leq\int|\nabla f|^2e^{-u}dv.
\end{equation}
By Perelman's estimate to $u$ (\ref{perelman bound:Ricci potential1}), $$\int f^2dv\leq C(g_0)\int|\nabla f|^2dv.$$
Thus,
$$\int|\nabla f|^2dv=-\int f\triangle fdv\leq\frac{1}{2C}\int f^2dv+2C\int(\triangle f)^2dv$$
from which it follows
$$\int|\nabla f|^2dv\leq C(g_0)\int(\triangle f)^2dv.$$

Doing integration by parts gives
\begin{eqnarray}
\int|\nabla\nabla f|^2dv&=&\int\big((\triangle f)^2-R_{i\bar{j}}\nabla_{\bar{i}}f\nabla_jf\big)dv\nonumber\\
&=&\int\big((\triangle f)^2-|\nabla f|^2+\nabla_i\nabla_{\bar{j}}u\nabla_{\bar{i}}f\nabla_jf\big)dv.\nonumber
\end{eqnarray}
The last term on the right hand side can be estimated as follows
\begin{eqnarray}
\int\nabla_i\nabla_{\bar{j}}u\nabla_{\bar{i}}f\nabla_jfdv&=&-\int\nabla_{\bar{j}}u\big(\triangle f\nabla_jf+\nabla_{\bar{i}}f\nabla_i\nabla_jf\big)dv\nonumber\\
&\leq&\int\big((\triangle f)^2+\frac{1}{2}|\nabla\nabla f|^2+\|\nabla u\|_{C^0}^2|\nabla f|^2\big)dv.\nonumber
\end{eqnarray}
Combining these estimates we have
\begin{equation}\nonumber
\int|\nabla\nabla f|^2dv\leq\int\big(4(\triangle f)^2+2C|\nabla f|^2\big)dv\leq C\int(\triangle f)^2dv\leq C\int|\nabla\bar{\nabla}f|^2dv,
\end{equation}
the desired estimate.
\end{proof}

Now we give the proof of the propositions:

\begin{proof}[Proof of the Proposition \ref{prop1}]
For any time $t=k\ge 1$, choose a normalized minimizer of $\mu(g(k))$, say $f_k$, satisfing $\int_Me^{-f_k}dv=\V$. Let $f_k(t)$ be the solution to (\ref{backheat}) on the time interval $[k-1,k]$. Then we have
\begin{equation}\nonumber
\frac{1}{\V}\int_{k-1}^k\int_M\big(|\nabla\bar{\nabla}(u-f_k)|^2+|\nabla\nabla f_k|^2\big)e^{-f_k}dvdt\leq\mu(g(k))-\mu(g(k-1)).
\end{equation}
It is proved that $|f_k(t)|\leq C(g_0)$ for any $t\in[k-1,k]$ (see \cite{TiZh12, TiZhu13, TZZZ}). Thus,
\begin{equation}\nonumber
\int_{k-1}^k\int_M\big(|\nabla\bar{\nabla}(u-f_k)|^2+|\nabla\nabla f_k|^2\big)dvdt\leq C(g_0)\big(\mu(g(k))-\mu(g(k-1))\big).
\end{equation}
Summing up $k=1,2,\cdots,$ and using $\mu(g(t))\leq 0$ for all $t$, we conclude
\begin{equation}\label{e301}
\sum_{k=1}^\infty\int_{k-1}^k\int_M\big(|\nabla\bar{\nabla}(u-f_k)|^2+|\nabla\nabla f_k|^2\big)dvdt\leq C(g_0).
\end{equation}
Applying Lemma \ref{lemma2} to $u-f_k$ we get
\begin{equation}\nonumber
\int|\nabla\nabla(u-f_k)|^2dv\leq C(g_0)\int|\nabla\bar{\nabla}(u-f_k)|^2dv,\hspace{0.5cm}\forall t\in[k-1,k].
\end{equation}
Combining with (\ref{e301}) it gives
\begin{equation}\nonumber
\int_0^\infty\int_M|\nabla\nabla u|^2dv\leq\sum_{k=1}^\infty\int_M\big(2|\nabla\nabla(u-f_k)|^2+2|\nabla\nabla f_k|^2\big)dv\leq C(g_0).
\end{equation}
This proves (\ref{L2bound:Ricci potential1}).

To prove (\ref{L2bound:Ricci potential2}), it will be sufficient to show
\begin{equation}\label{e302}
\frac{d}{dt}\int|\nabla\nabla u|^2dv\leq C(g_0),\hspace{0.5cm}\forall t\geq 0.
\end{equation}
An easy calculation shows
\begin{equation}
\frac{\partial}{\partial t}|\nabla\nabla u|^2=\triangle|\nabla\nabla u|^2-|\bar{\nabla}\nabla\nabla u|^2-|\nabla\nabla\nabla u|^2-2R_{i\bar{j}k\bar{l}}\nabla_{\bar{i}}\nabla_{\bar{k}}u\nabla_j\nabla_lu.
\end{equation}
Integrating this formula gives
\begin{eqnarray}
\frac{d}{dt}\int|\nabla\nabla u|^2dv&=&\int\big(-|\bar{\nabla}\nabla\nabla u|^2-|\nabla\nabla\nabla u|^2+\triangle u|\nabla\nabla u|^2\nonumber\\
&&\hspace{3cm}-2R_{i\bar{j}k\bar{l}}\nabla_{\bar{i}}\nabla_{\bar{k}}u\nabla_j\nabla_lu\big)dv
\nonumber\\
&\leq&\int\big(\triangle u|\nabla\nabla u|^2+2\nabla_{\bar{j}}R_{k\bar{l}}\nabla_{\bar{k}}u\nabla_j\nabla_lu
+2R_{i\bar{j}k\bar{l}}\nabla_{\bar{k}}u\nabla_{\bar{i}}\nabla_j\nabla_lu\big)dv
\nonumber\\
&\leq&\int\bigg((\|\triangle u\|_{C^0}+\|\nabla u\|_{C^0}^2)|\nabla\nabla u|^2+|\nabla\nabla\bar{\nabla} u|^2+\|\nabla u\|_{C^0}^2|Rm|^2\bigg)dv.\nonumber
\end{eqnarray}
The desired estimate (\ref{e302}) now follows from Perelman's estimate to $u$ (\ref{perelman bound:Ricci potential1}) and the general estimates (\ref{integral bound:1}), (\ref{integral bound:4}) in next section.
\end{proof}

\begin{proof}[Proof of proposition \ref{prop2}]
First observe that, by using Ricci potential equation,
$$\nabla_i(\triangle u-|\nabla u|^2+u)=\nabla_{\bar{j}}\nabla_i\nabla_ju-\nabla_i\nabla_ju\nabla_{\bar{j}}u.$$
Thus,
$$|\nabla(\triangle u-|\nabla u|^2+u)|^2\leq 2(|\nabla u|^2|\nabla\nabla u|^2+|\nabla_{\bar{j}}\nabla_i\nabla_ju|^2).$$
To prove (\ref{L2bound:Ricci potential3}) it suffices to show $\int_t^{t+1}\int|\nabla_{\bar{j}}\nabla_i\nabla_ju|^2dvdt\rightarrow 0$.

Integrating by parts and using the second Bianchi identity,
\begin{eqnarray}
\int|\nabla_{\bar{j}}\nabla_i\nabla_ju|^2dv&=&\int\nabla_{\bar{j}}\nabla_i\nabla_ju\nabla_k\nabla_{\bar{i}}\nabla_{\bar{k}}udv\nonumber\\
&=&-\int\nabla_i\nabla_ju\nabla_{\bar{j}}\big(\nabla_{\bar{i}}\triangle u+R_{\bar{i}k}\nabla_{\bar{k}}u\big)dv\nonumber\\
&=&-\int\nabla_i\nabla_ju\big(\nabla_{\bar{j}}\nabla_{\bar{i}}\triangle u+\nabla_{\bar{j}}R_{\bar{i}k}\nabla_{\bar{k}}u
+R_{\bar{i}k}\nabla_{\bar{j}}\nabla_{\bar{k}}u\big)dv\nonumber\\
&=&-\int\nabla_i\nabla_ju\nabla_{\bar{j}}\nabla_{\bar{i}}\triangle udv+\int\nabla_i\nabla_ju\nabla_{\bar{j}}\nabla_{\bar{i}}\nabla_ku\nabla_{\bar{k}}udv\nonumber\\
&&-\int|\nabla\nabla u|^2dv+\int\nabla_i\nabla_ju\nabla_{\bar{i}}\nabla_ku\nabla_{\bar{j}}\nabla_{\bar{k}}udv\nonumber\\
&=&-\int\nabla_i\nabla_ju\nabla_{\bar{j}}\nabla_{\bar{i}}\triangle udv+\int\nabla_i\nabla_ju\nabla_{\bar{j}}\nabla_{\bar{i}}\nabla_ku\nabla_{\bar{k}}udv\nonumber\\
&&-\int|\nabla\nabla u|^2dv-\int\nabla_ku\big(\nabla_i\nabla_ju\nabla_{\bar{i}}\nabla_{\bar{j}}\nabla_{\bar{k}}u
+\nabla_{\bar{i}}\nabla_i\nabla_ju\nabla_{\bar{j}}\nabla_{\bar{k}}u\big)\nonumber.
\end{eqnarray}
Then, by Schwarz inequality,
\begin{eqnarray}
\int_{t-1}^t\int|\nabla_{\bar{j}}\nabla_i\nabla_ju|^2dvdt&\leq&\big(\int_{t-1}^t\int|\nabla\nabla u|^2dvdt\big)^{1/2}\cdot
\bigg[\big(\int_{t-1}^t\int|\nabla\nabla\triangle u|^2dvdt\big)^{1/2}\nonumber\\
&&+\big(\int_{t-1}^t\int|\nabla u|^2(|\nabla\nabla\nabla u|^2+|\nabla\nabla\bar{\nabla}u|^2+|\bar{\nabla}\nabla\nabla u|^2)\big)^{1/2}\bigg].\nonumber
\end{eqnarray}
Applying (\ref{integral bound:4}) and the $L^2$ bound of $\nabla\nabla\triangle u$ (see Remark \ref{remark201}) we get (\ref{L2bound:Ricci potential3}).

Set $h=\triangle u-|\nabla u|^2+u-a$. Noticing that $\int he^{-u}dv=0$, by weighted Poincar\'e inequality, and using uniform bound of $u$, we derive
\begin{equation}\nonumber
\int h^2dv\leq C(g_0)\int|\nabla(\triangle u-|\nabla u|^2+u)|^2dv.
\end{equation}
Thus,
\begin{equation}\nonumber
\int_t^{t+1}\int h^2dvdt\rightarrow 0,\hspace{0.5cm}\mbox{ as }t\rightarrow\infty.
\end{equation}
To show (\ref{L2bound:Ricci potential4}), it suffices to prove
$$\frac{d}{dt}\int h^2dv\leq C(g_0)(1+\int h^2dv).$$
Actually, by
\begin{equation}\nonumber
\frac{\partial}{\partial t}h=\triangle h+h-\frac{d}{dt}a+|\nabla_i\nabla_ju|^2,
\end{equation}
we have
\begin{eqnarray}\nonumber
\frac{d}{dt}\int h^2dv&=&\int2h(\triangle h+h-\frac{d}{dt}a+|\nabla_i\nabla_ju|^2+\frac{1}{2}h\triangle u)dv\\
&\leq&\int2h(h+|\nabla_i\nabla_ju|^2+\frac{1}{2}h\triangle u)dv\nonumber\\
&\leq&(3+\|\triangle u\|_{C^0})\int h^2dv+\int|\nabla\nabla u|^4dv.\nonumber
\end{eqnarray}
The required estimate follows from (\ref{integral bound:4}) in next section.
\end{proof}

\subsection{Regularity of the limit}

For any sequence $t_i\rightarrow\infty$, define a family of K\"ahler-Ricci flows
\begin{equation}
(M,g_i(t))\,=\,(M,g(t_i+t)),\,t\geq -1.
\end{equation}
Let $u_i(t)$ denote associated Ricci potentials, which satisfy the uniform bound
\begin{equation}\label{perelman bound:Ricci potential2}
\|u_i\|_{C^0}+\|\nabla u_i\|_{C^0}+\|\triangle u_i\|_{C^0}\,\le\, C(g_0).
\end{equation}
Furthermore, by (\ref{L2bound:Ricci potential2}), for any $t\geq -1$,
\begin{equation}\label{limit0}
\int_M\big|\nabla\nabla u_i(t)\big|^2dv_{g_i}\rightarrow0,\,\mbox{ as }i\rightarrow\infty.
\end{equation}

By the convergence theorem in Section 2, passing to a subsequence if necessary we may assume at time $t=0$,
\begin{equation}\label{limit1}
(M,g_i(0))\stackrel{d_{GH}}{\longrightarrow}(M_\infty,d).
\end{equation}
The space $M_\infty=\mathcal{S}\cup\mathcal{R}$, where $\mathcal{R}$ is a smooth complex manifold with a $C^{\alpha}$ complex structure $J_\infty$ and a $C^\alpha$ metric $g_\infty$ which induces $d$, while $\mathcal{S}$ is a closed singular set of codimension $\geq 4$. Moreover, under the Gromov-Hausdorff convergence,
\begin{equation}\label{L2pconvergence}
(g_i(0),u_i(0))\stackrel{C^\alpha\cap L^{2,p}}{\longrightarrow}(g_\infty,u_\infty)\,\mbox{ on }\mathcal{R}.
\end{equation}
The convergence of $u_i(0)$ follows from the elliptic regularity to $\triangle u_i(0)=n-s(g_i(0))\in L^p$. It is obvious $u_\infty$ is globally Lipschitz on $M_\infty$ by Perelman's estimate.

\begin{prop}
Suppose {\rm(\ref{L2pconvergence})} holds, then $g_\infty$ is smooth and satisfies
\begin{equation}\label{limitsoliton1}
Ric(g_\infty)+\Hess u_\infty\,=\,g_\infty,\,\mbox{ on }\mathcal{R}.
\end{equation}
Moreover, $J_\infty$ is smooth on $\mathcal{R}$ and $g_\infty$ is K\"ahler with respect to $J_\infty$.
\end{prop}
\begin{proof}
We first show $g_\infty$ is smooth and satisfies (\ref{limitsoliton1}). The strategy is to apply a bootstrap as in \cite{Pe96}; the difference is the existence of a twisted function term. In local harmonic coordinate $(x^1,\cdots,x^{2n})$, the soliton equation (\ref{limitsoliton1}) is equivalent to
\begin{equation}\label{limitsoliton2}
g^{\alpha\beta}\frac{\partial^2g_{\gamma\delta}}{\partial x^\alpha\partial x^\beta}\,=\,-\frac{\partial^2u_\infty}{\partial x^\gamma\partial x^\delta}+Q(g,\partial g)_{\gamma\delta}+T(g^{-1},\partial g,\partial u)_{\gamma\delta}+g_{\gamma\delta},
\end{equation}
where $Q$ is a quadratical term while $T$ is a trilinear term of their variables.
By Proposition \ref{prop1}, (\ref{limitsoliton2}) holds in $L^2(\mathcal{R})$. Since both $u_\infty$ and $g_\infty$ are in $L^{2,p}$, equation (\ref{limitsoliton2}) holds even in $L^p(\mathcal{R})$. On the other hand, by Proposition \ref{prop2} and (\ref{L2pconvergence}), we have that,
\begin{equation}\label{limitsoliton3}
g^{\alpha\beta}\frac{\partial^2 u_\infty}{\partial x^\alpha\partial x^\beta}\,=\,g^{\alpha\beta}\frac{\partial u_\infty}{\partial x^\alpha}\frac{\partial u_\infty}{\partial x^\beta}-2u_\infty+2a_\infty,
\end{equation}
in $L^p$ topology. A bootstrap argument to the elliptic systems (\ref{limitsoliton3}) and (\ref{limitsoliton2}) shows that $g_\infty$ and $u_\infty$ are actually smooth on $\mathcal{R}$.

Since $g_\infty$ is smooth and $\nabla_{g_\infty}J_\infty=0$, the elliptic regularity shows that $J_\infty$ is also smooth.
\end{proof}

\subsection{Smooth convergence on the regular set}

In order to prove the smooth convergence of the K\"ahler-Ricci flow on the regular set, we need the following version of Perelman's pseudolocality theorem: there exists $\epsilon_P,\delta_P>0$ and $r_P>0$, which depend on $p,\Lambda$ in the Theorem \ref{regularity:1} such that for any space-time point $(x_0,t_0)\in M\times[-1.\infty)$ in any flow $g_i(t)$ constructed in the previous subsection, if
\begin{equation}\label{pseudolocality1}
\vol_{g_i(t_0)}\big(B_{g_i(t_0)}(x_0,r)\big)\,\ge\,(1-\epsilon_P)\vol(B_r)
\end{equation}
for some $r\leq r_P$, where $\vol(B_r)$ denotes the volume of Euclidean ball of radius $r$ in $\mathbb{R}^{2n}$, then we have the following curvature estimate
\begin{equation}\label{pseudolocality2}
|Rm_{g_i}(x,t)|\,\le\,\frac{1}{t-t_0},\,\forall x\in B_{g_i(t)}(x_0,\epsilon_Pr),\,t_0<t\leq t_0+\epsilon_P^2r^2,
\end{equation}
and volume estimate
\begin{equation}\label{pseudolocality3}
\vol_{g_i(t)}\big(B_{g_i(t)}(x_0,\delta_P\sqrt{t-t_0})\big)\,\ge\,(1-\eta)\vol\big(B_{\delta_P\sqrt{t-t_0}}\big),\,t_0<t\leq t_0+\epsilon_P^2r^2
\end{equation}
where $\eta>0$ is the constant in (\ref{harmonic coordinate:3}). One can assume $\eta\le\epsilon_P$ in application. In other words, in view of Shi's higher derivative estimate to curvature \cite{Sh89}, the region around $x_0$ is almost Euclidean in the $C^\infty$ topology at time $t_0+\epsilon_P^2r^2$.

Notice that Perelman's pseudolocality theorem is originally stated for Ricci flow \cite{Pe02}. In our application, the sequence of K\"ahler flows $g_i(t)$ comes from a Ricci flow by scalings with a definite control (by Perelman's estimate to scalar curvature (\ref{perelman bound:Ricci potential1})). The condition (\ref{pseudolocality1}) implies the local $C^\alpha$ structure at $x_0$ ( see Section 2.6).

We start to prove the part of smooth convergence in our Main Theorem.

\begin{proof}[Proof of Theorem \ref{regularity:1}]
Recall that we have a family of flows $(M,g_i(t)),-1\le t<\infty$, which converges at $t=0$ in the Cheeger-Gromov topology to a limit space $(M_\infty,d)$. The regular set $\mathcal{R}$ is a smooth complex manifold with a smooth metric $g_\infty$ which induces $d$; the singular set $\mathcal{S}$ is closed and has codimension $\geq 4$. Moreover, the metric $g_i(0)$ converges to $g_\infty$ in the $C^\alpha$ sense on $\mathcal{R}$. The goal is to show that $g_i(0)$ converges smoothly to $g_\infty$.

For any radius $0<r\le r_P$, integer $i$ and time $t\ge-1$ let us define
$$K_{r,i,t}=\{x\in M|{\rm(\ref{pseudolocality1})}\mbox{ holds on the ball }B_{g_i(t)}(x,r).\}$$
Then (\ref{pseudolocality3}) implies
\begin{equation}\label{pseudolocality4}
K_{r,i,t}\subset K_{\delta_P\sqrt{s},i,t+s},\,\forall i,0<s\le \epsilon_P^2r^2.
\end{equation}

First of all, observe that by the volume continuity under the Cheeger-Gromov convergence there exists for any $\varepsilon>0$, $i\ge 1$ and $-1\le t_0\le 0$ a radius $0<r_\varepsilon\le r_P$ such that
\begin{equation}\label{pseudolocality5}
\vol_{g_i(t_0)}(M\backslash K_{r_\varepsilon,i,t_0})\,\le\,\varepsilon.
\end{equation}

Now, let $r_j$ be a decreasing sequence of radii such that $\lim_{j\rightarrow\infty}r_j=0$ and $t_j=-\epsilon_Pr_j$ be a sequence of time. (\ref{pseudolocality5}) implies that
\begin{equation}\label{pseudolocality6}
\vol_{g_i(t)}(M\backslash K_{r_j,i,t})\rightarrow 0
\end{equation}
uniformly as $j\rightarrow \infty$. By (\ref{pseudolocality6}), after a suitable adjusting of the radii $r_j$ we may assume
\begin{equation}\label{pseudolocality6.5}
K_{r_j,i,t_j}\subset K_{r_{j+1},i,t_{j+1}},\,\forall i,j\ge 1.
\end{equation}
By pseudolocality theorem,
\begin{equation}\label{pseudolocality7}
|Rm(g_i(t))(x)|\,\le\,(t-t_j)^{-1},
\end{equation}
for all $(x,t)$ satisfying
$$d_{g_i(t)}(x,K_{r_j,i,t_j})\,\le\,\epsilon_Pr_j,\,t_j<t\le 0.$$
By Shi's derivative estimate to the curvature under Ricci flow \cite{Sh89}, there exist a sequence of constants $C_{k,j,i}$ such that
\begin{equation}\label{pseudolocality8}
|(\nabla^L)^kRm(g_i(0))|\,\le\, C_{k,j,i},\,\mbox{ on }K_{r_j,i,t_j}
\end{equation}
where $\nabla^L$ denotes the Levi-Civita connection of the corresponding Riemannian metric. Passing to a subsequence of $\{j\}$ if necessary one can find a subsequence $\{i_j\}$ of $\{i\}$ such that
\begin{equation}\label{pseudolocality9}
(K_{r_j,i_j,t_j},g_{i_j}(t_j))\,\stackrel{C^\alpha}{\longrightarrow}(\Omega^{'},g_{\Omega^{'}})
\end{equation}
and
\begin{equation}\label{pseudolocality10}
(K_{r_j,i_j,t_j},g_{i_j}(0))\,\stackrel{C^\infty}{\longrightarrow}(\Omega,g_\Omega)
\end{equation}
where $(\Omega^{'},g_{\Omega^{'}})$ and $(\Omega,g_\Omega)$ are smooth Riemannian manifolds (may not be complete). We can also assume
$$\big(M,g_{i_j}(t_j)\big)\,\stackrel{d_{GH}}{\longrightarrow}\,(M_\infty^{'},d^{'})$$
where $(M_\infty^{'},d^{'})$ is a complete length space as described in Section 2. Let $M_\infty^{'}=\mathcal{R}^{'}\cup\mathcal{S}^{'}$ be the regular-singular decomposition of $M_\infty^{'}$ and $g_\infty^{'}$ be a Riemannian metric on $\mathcal{R}^{'}$ which induces $d^{'}$. Then we have the following two claims.

\begin{clai}\label{claim:1}
$(\Omega^{'},g_{\Omega^{'}})$ is isometric to $(\mathcal{R}_\infty^{'},g_\infty^{'})$.
\end{clai}

\begin{clai}\label{claim:2}
$(\Omega^{'},g_{\Omega^{'}})$ is isometric to $(\Omega,g_\Omega)$.
\end{clai}

The smooth convergence of $g_i(0)$ to $g_\infty$ on $\mathcal{R}$ will follow directly from (\ref{pseudolocality10}) once the two claims are proved.

\begin{proof}[Proof of Claim \ref{claim:1}]
Obviously $(\Omega^{'},g_{\Omega^{'}})$ can be viewed as a subset of $(\mathcal{R}_\infty^{'},g_\infty^{'})$. To show the equality, just notice that a point of $M_\infty^{'}$ belongs to $\mathcal{R}_\infty^{'}$ iff there is a local $C^\alpha$ structure around it and then by the continuity of volume under the Cheeger-Gromov convergence this point is a limit of points in $K_{r_j,i_j,t_j}$.
\end{proof}

\begin{proof}[Proof of Claim \ref{claim:1}]
Using the curvature estimate (\ref{pseudolocality7}), by a similar argument as in the proof of Lemma 8.3 (b) of \cite{Pe02} we can show that for any endpoints $x,y\in K_{r_j,i_j,t_j}$,
$$\frac{d}{dt}d_{g_{i_j}(t)}(x,y)\,\ge\,-C(n,g_0)\cdot(t-t_j)^{-\frac{1}{2}},\,\forall t_j<t\le t_j+\epsilon_P^2r_j^2.$$
Integrating from time $t_j$ to $0$ gives
\begin{equation}\label{pseudolocality11}
d_{g_{i_j}(0)}(x,y)\,\ge\,d_{g_{i_j}(t_j)}(x,y)-C(n,g_0)\cdot\sqrt{-t_j},\,\forall x,y\in K_{r_j,i_j,t_j}.
\end{equation}
Passing to the limit we get an expanding map
$$\psi_\infty:\mathcal{R}^{'}\rightarrow\Omega.$$
Since the volume
$$\vol_{g_\Omega}(\Omega)\,\le\,\V\,=\,\vol_{g_\infty^{'}}(\mathcal{R}^{'}),$$
the map $\psi_\infty$ must be an isometry.
\end{proof}

The proof is now complete.
\end{proof}


\section{$L^4$ bound of Ricci curvature under K\"ahler-Ricci flow}

In this section, we prove that the $L^4$ norm of Ricci curvature is uniformly bounded along the K\"ahler-Ricci flow. This is
crucial in the application of our regularity theory established in the previous section to the Hamilton-Tian conjecture on low dimensional manifolds.

Let $M$ be a compact Fano $n$-manifold and $g_0$ be a K\"ahler metric in the class $2\pi c_1(M;J)$. Let $g(t)$ be the solution to the volume normalized K\"ahler-Ricci flow (\ref{KRF}) with initial $g(0)=g_0$.

\begin{theo}
There exists a constant $C=C(g_0)$ such that
\begin{equation}\label{L4Ricci bound}
\int_M|Ric(g(t))|^4dv_{g(t)}\,\leq\, C,\,\forall t\geq 0.
\end{equation}
\end{theo}

Let $u(t)$ be the Ricci potential of $g(t)$, which are determined by (\ref{Ricci potential}). Then (\ref{L4Ricci bound}) is equivalent to the uniform $L^4$ bound of $\nabla\bar{\nabla}u$. The proof relies on the heat equation
\begin{equation}
\frac{\partial}{\partial t}u=\triangle u+u-a
\end{equation}
where $a=a(t)$ is a family of constants defined by (\ref{average potential}), as well as the uniform $L^2$ bound of total Riemannian curvature. An elliptic version of such integral Hessian estimate and its application are discussed by Cheeger in \cite{Ch03}. We remark that it is much more subtle in our parabolic case under the K\"ahler-Ricci flow. Actually, Perelman's gradient estimate and Laplacian estimate to $u$ will be used essentially in the proof.

Recall Perelman's estimate,
\begin{equation}\label{perelman bound:Ricci potential3}
\|u(t)\|_{C^0}+\|\nabla u(t)\|_{C^0}+\|\triangle u(t)\|_{C^0}\,\le\, C
\end{equation}
where $C=C(g_0)$. To prove the $L^4$ bound of $\Hess(u)$, we need several lemmas.

\begin{lemm}
There exists $C=C(g_0)$ such that
\begin{equation}\label{integral bound:1}
\int\big(|\nabla\nabla u|^2+|\nabla\bar{\nabla}u|^2+|Rm|^2\big)dv\,\le\, C,\,\forall t\ge 0.
\end{equation}
\end{lemm}
\begin{proof}
The $L^2$ bound of $\nabla\bar{\nabla}u$ follows from the observation
$$\int|\nabla\bar{\nabla}u|^2\,=\,\int(\triangle u)^2.$$
The $L^2$ bound of $\nabla\nabla u$ follows from an integration by parts:
\begin{eqnarray}
\int|\nabla\nabla u|^2dv&=&\int\big((\triangle u)^2-R_{i\bar{j}}\nabla_{\bar{i}}u\nabla_ju\big)dv\nonumber\\
&=&\int\big((\triangle u)^2-|\nabla u|^2+\nabla_i\nabla_{\bar{j}}u\nabla_{\bar{i}}u\nabla_ju\big)dv\nonumber\\
&\leq&\int\big((\triangle u)^2+|\nabla\bar{\nabla}u|^2+|\nabla u|^4\big)dv.\nonumber
\end{eqnarray}

The $L^2$ bound of Riemannian curvature tensor follows from the Chern-Weil theory. Denote by $c_i$ the i-th Chern class. Let $W$ be the Weyl tensor and define $U$ and $Z$ as
$$U\,=\,\frac{s}{2n(2n-1)}g\odot g,\,Z\,=\,\frac{1}{2n-2}\big(Ric-\frac{s}{n}g\big)\odot g$$
where $s$ is the scalar curvature of $g$, $\odot$ is the Kulkarni-Nomizu product. Then we have the general formula \cite[Page 80]{Be}
$$\int c_2\wedge c_1^{n-2}\,=\,\frac{(n-2)!}{2(2\pi)^n}\int\big((2n-3)(n-1)|U|^2-(2n-3)|Z|^2+|W|^2\big).$$
The $L^2$ norms of $Z$ and $U$ are uniformly bounded in terms of $\int|\nabla\bar{\nabla}u|^2dv$. Since the left hand side of above formula is a topological invariant, the $L^2$ norm of Weyl tensor is uniformly bounded and this in turn gives the uniform $L^2$ bound of the total curvature tensor.
\end{proof}

\begin{lemm}
There exists a constant $C=C(g_0)$ such that
\begin{equation}\label{integral bound:2a}
\int|\nabla\bar{\nabla}u|^4dv\,\le\, C\int\big(|\nabla\nabla\bar{\nabla}u|^2+|\bar{\nabla}\nabla\nabla u|^2\big)dv,\,\forall t,
\end{equation}
and
\begin{equation}\label{integral bound:2b}
\int|\nabla\nabla u|^4dv\,\le\, C\int\big(|\nabla\nabla\nabla u|^2+|\nabla\nabla\bar{\nabla}u|^2+|\bar{\nabla}\nabla\nabla u|^2\big)dv,\,\forall t.
\end{equation}
\end{lemm}
\begin{proof}
Recall the Bochner formula,
$$\triangle|\nabla u|^2=|\nabla\nabla u|^2+|\nabla\bar{\nabla}u|^2+\triangle\nabla_iu\nabla_{\bar{i}}u+\nabla_iu\triangle\nabla_{\bar{i}}u.$$
Multiplying $|\nabla\bar{\nabla}u|^2$ and integrating over $M$ we have
\begin{eqnarray}
&&\int\big(|\nabla\nabla u|^2+|\nabla\bar{\nabla}u|^2\big)|\nabla\bar{\nabla}u|^2\nonumber\\
&=&\int\triangle|\nabla u|^2|\nabla\bar{\nabla}u|^2
-\triangle\nabla_iu\nabla_{\bar{i}}u|\nabla\bar{\nabla}u|^2-\nabla_iu\triangle\nabla_{\bar{i}}u|\nabla\bar{\nabla}u|^2\nonumber\\
&=&-\int\big(\nabla_{\bar{i}}\nabla_ju\nabla_{\bar{j}}u+\nabla_{\bar{i}}\nabla_{\bar{j}}u\nabla_ju\big)
\big(\nabla_i\nabla_j\nabla_{\bar{k}}u\nabla_{\bar{j}}\nabla_ku+\nabla_i\nabla_{\bar{j}}\nabla_ku\nabla_j\nabla_{\bar{k}}u\big)\nonumber\\
&&-\int\big(\triangle\nabla_iu\nabla_{\bar{i}}u|\nabla\bar{\nabla}u|^2+\nabla_iu\triangle\nabla_{\bar{i}}u|\nabla\bar{\nabla}u|^2\big)\nonumber\\
&\leq&2\int|\nabla u||\nabla\bar{\nabla}u||\nabla\nabla\bar{\nabla}u|\big(|\nabla\bar{\nabla}u|+|\nabla\nabla u|\big)\nonumber\\
&&+n\int|\nabla u||\nabla\bar{\nabla}u|^2\big(|\nabla\nabla\bar{\nabla}u|+|\bar{\nabla}\nabla\nabla u|\big)\nonumber\\
&\leq&\frac{1}{2}\int|\nabla\bar{\nabla}u|^2\big(|\nabla\bar{\nabla}u|^2+|\nabla\nabla u|^2\big)+8(n^2+1)\int|\nabla u|^2\big(|\nabla\nabla\bar{\nabla}u|^2+|\bar{\nabla}\nabla\nabla u|^2\big).\nonumber
\end{eqnarray}
This gives the first estimate by Perelman's estimate (\ref{perelman bound:Ricci potential3}). The second estimate can be derived similarly.
\end{proof}

\begin{lemm}\label{lemma1}
There exists a constant $C=C(g_0)$ such that
\begin{eqnarray}\label{integral bound:3}
&&\int_M\big(|\bar{\nabla}\nabla\nabla u|^2+|\nabla\nabla\bar{\nabla}u|^2+|\nabla\nabla\nabla u|^2\big)dv\nonumber\\
&&\hspace{3cm}\leq C\int_M\big(|\nabla\triangle u|^2+|Rm|^2+|\nabla\nabla u|^2\big)dv,\,\forall t.
\end{eqnarray}
\end{lemm}
\begin{proof}
We prove the estimates by integration by parts. For example,
\begin{eqnarray}
\int|\nabla\nabla\bar{\nabla}u|^2dv&=&\int\nabla_i\nabla_j\nabla_{\bar{k}}u\nabla_{\bar{i}}\nabla_{\bar{j}}\nabla_kudv\nonumber\\
&=&\int\nabla_i\nabla_j\nabla_{\bar{k}}u(\nabla_k\nabla_{\bar{i}}\nabla_{\bar{j}}u+R_{\bar{i}kl\bar{j}}\nabla_{\bar{l}}u)dv\nonumber\\
&=&\int\big(-\nabla_i\nabla_j\triangle
u\nabla_{\bar{i}}\nabla_{\bar{j}}u+R_{\bar{i}kl\bar{j}}\nabla_{\bar{l}}u\nabla_i\nabla_j\nabla_{\bar{k}}u)\big)dv\nonumber\\
&=&\int\big(\nabla_j\triangle u(\nabla_{\bar{j}}\triangle
u+R_{i\bar{j}}\nabla_{\bar{i}}u)+R_{\bar{i}kl\bar{j}}\nabla_{\bar{l}}u\nabla_i\nabla_j\nabla_{\bar{k}}u)\big)dv\nonumber\\
&\leq&\int\big(\frac{1}{2}|\nabla\nabla\bar{\nabla}u|^2+2|\nabla\triangle u|^2+n|\nabla u|^2|Rm|^2\big)dv\nonumber,
\end{eqnarray}
then apply Perelman's estimate (\ref{perelman bound:Ricci potential3}) to get the desired estimate of $\int|\nabla\nabla\bar{\nabla}u|^2dv$. The estimate of $\int|\bar{\nabla}\nabla\nabla u|^2dv$ is totally same as $\int|\nabla\nabla\bar{\nabla}u|^2dv$. The additional term $\int|\nabla\nabla u|^2dv$ on the right hand side of (\ref{integral bound:3}) comes from the integration by parts in proving the estimate of $\int|\nabla\nabla\nabla u|^2dv$.
\end{proof}

Now we are ready to prove the $L^4$ bound of $\Hess(u)$ under the K\"ahler-Ricci flow.

\begin{theo}
There exists $C=C(g_0)$ such that,
\begin{equation}\label{integral bound:4}
\int\big(|\nabla\nabla\bar{\nabla}u|^2+|\bar{\nabla}\nabla\nabla u|^2+|\nabla\nabla\nabla u|^2+|\nabla\nabla u|^4+|\nabla\bar{\nabla}u|^4\big)dv\,\le\, C,\,\forall t\ge0.
\end{equation}
\end{theo}
\begin{proof}
It is sufficient to show a uniform bound of $\int|\nabla\triangle u|^2$ under the K\"ahler-Ricci flow. To this purpose consider the evolution of $(\triangle u)^2$:
\begin{eqnarray}\nonumber
\frac{\partial}{\partial t}(\triangle u)^2\,=\,\triangle(\triangle u)^2-2|\nabla\triangle u|^2-2\triangle u|\nabla\bar{\nabla}u|^2+2(\triangle u)^2.
\end{eqnarray}
Integrating this identity gives
\begin{eqnarray}
2\int|\nabla\triangle u|^2&=&\int\big(-2\triangle u|\nabla\bar{\nabla}u|^2+2(\triangle u)^2-\frac{\partial}{\partial t}(\triangle u)^2\big)\nonumber\\
&=&\int\big(-2\triangle u|\nabla\bar{\nabla}u|^2+2(\triangle u)^2+(\triangle u)^3\big)-\frac{d}{d t}\int(\triangle u)^2\nonumber.
\end{eqnarray}
Together with Perelman's estimate (\ref{perelman bound:Ricci potential3}) and (\ref{integral bound:1}) this implies
\begin{equation}\label{integral bound:5}
\int_t^{t+1}\int|\nabla\triangle u|^2\leq C(g_0),\hspace{0.5cm}\forall t\geq 0.
\end{equation}

Next we calculate the derivative
\begin{eqnarray}
\frac{d}{dt}\int|\nabla\triangle u|^2&=&\int\big(-|\nabla\nabla\triangle u|^2-|\nabla\bar{\nabla}\triangle u|^2+|\nabla\triangle u|^2\nonumber\\
&&-\nabla_i|\nabla\bar{\nabla}u|^2\nabla_{\bar{i}}\triangle u-\nabla_{\bar{i}}|\nabla\bar{\nabla}u|^2\nabla_i\triangle u+\triangle u|\nabla\triangle u|^2\big)\nonumber.
\end{eqnarray}
By integration by parts,
\begin{eqnarray}
\bigg|\int-\nabla_i|\nabla\bar{\nabla}u|^2\nabla_{\bar{i}}\triangle u\bigg|&=&\bigg|-\int\big(\nabla_i\nabla_j\nabla_{\bar{k}}u\nabla_{\bar{j}}\nabla_ku
+\nabla_j\nabla_{\bar{k}}u\nabla_i\nabla_{\bar{j}}\nabla_ku\big)\nabla_{\bar{i}}\triangle u\bigg|\nonumber\\
&\leq&\bigg|\int\nabla_{\bar{j}}u\big(\nabla_i\nabla_j\triangle u\nabla_{\bar{i}}\triangle u+\nabla_i\nabla_j\nabla_{\bar{k}}u\nabla_k\nabla_{\bar{i}}\triangle u\big)\bigg|\nonumber\\
&&+\bigg|\int\nabla_{\bar{k}}u\big(\nabla_i\nabla_k\triangle u\nabla_{\bar{i}}\triangle u+\nabla_i\nabla_{\bar{j}}\nabla_ku\nabla_j\nabla_{\bar{i}}\triangle u\big)\bigg|\nonumber\\
&\leq&\frac{1}{4}\int\big(|\nabla\nabla\triangle u|^2+|\nabla\bar{\nabla}\triangle u|^2\big)\nonumber\\
&&+C(g_0)\int\big(|\nabla\triangle u|^2+|\nabla\nabla\bar{\nabla}u|^2\big),\nonumber
\end{eqnarray}
and, similarly
\begin{eqnarray}
\bigg|\int-\nabla_i\triangle u\nabla_{\bar{i}}|\nabla\bar{\nabla}u|^2\bigg|&\leq&\frac{1}{4}\int\big(|\nabla\nabla\triangle u|^2+|\nabla\bar{\nabla}\triangle u|^2\big)\nonumber\\
&&+C(g_0)\int\big(|\nabla\triangle u|^2+|\nabla\nabla\bar{\nabla}u|^2\big).\nonumber
\end{eqnarray}
Thus, by (\ref{integral bound:3}) and Perelman's estimate (\ref{perelman bound:Ricci potential3}) we obtain
\begin{eqnarray}\nonumber
\frac{d}{dt}\int|\nabla\triangle u|^2&\leq&-\frac{1}{2}\int\big(|\nabla\nabla\triangle u|^2+|\nabla\bar{\nabla}\triangle u|^2\big)+C(g_0)\big(1+\int|\nabla\triangle u|^2\big)\nonumber\\
&\leq&C(g_0)\big(1+\int|\nabla\triangle u|^2\big)\nonumber.
\end{eqnarray}
Together with (\ref{integral bound:5}) this implies the uniform bound of $\int|\nabla\triangle u|^2$. The proof is complete.
\end{proof}

\begin{rema}\label{remark201}
From the proof of above theorem, we also have that, under the K\"ahler-Ricci flow
\begin{equation}\label{integral bound:6}
\int_t^{t+1}\int(|\nabla\nabla\triangle u|^2+|\nabla\bar{\nabla}\triangle u|^2)\leq C(g_0),\,\forall t\ge 0.
\end{equation}
\end{rema}


\section{Proof of Theorem \ref{regularity:2}}

In this section, we shall prove Theorem 1.6 by generalizing the partial $C^0$-estimate to
K\"ahler-Ricci flow on Fano manifolds.

As before, let $M$ be a Fano manifold and $g(t)$ be a solution to the K\"ahler-Ricci flow (\ref{KRF}) in the canonical class $2\pi c_1(M)$ with initial $g(0)=g_0$. Let $u(t)$ be the Ricci potential of $g(t)$ defined by (\ref{Ricci potential}). Then we have Perelman's estimate (\ref{perelman bound:Ricci potential1})
\begin{equation}\label{perelman bound:Ricci potential4}
\|u(t)\|_{C^0}+\|\nabla u(t)\|_{C^0}+\|\triangle u(t)\|_{C^0}\leq C(g_0),\,\forall t\geq 0.
\end{equation}

Let $\tilde{g}(t)=e^{-\frac{1}{n}u(t)}g(t)$ and $h(t)$ be the induced metric of $\tilde{g}(t)$ on $K_M^{-\ell}$, the $\ell$-th power of the anti-canonical bundle ($\ell\geq 1$). Let $D$ be the Chern connection of $h(t)$. For simplicity, we also use $\nabla$ and $\bar{\nabla}$ to denote $\nabla\otimes D$ and $\bar{\nabla}\otimes D$ on the $K_M^{-\ell}$-valued tensor fields. The rough Laplacian on tensor fields is $\triangle=g^{i\bar{j}}\nabla_{\frac{\partial}{\partial z^i}}\nabla_{\frac{\partial}{\partial\bar{z}^j}}$.

Under these notations, the curvature form of the Chern connection
\begin{equation}\label{e17}
Ric(h(t))=\ell\omega(t),
\end{equation}
where $\omega(t)$ is the K\"ahler form of $g(t)$.

Set $N_\ell=\dim H^0(M,K_M^{-\ell}) - 1 $, where $\ell\geq1$. At any time $t$, we choose an orthonormal basis $\{s_{t,l,k}\}_{k=0}^{N_{\ell}}$ of $H^0(M,K_M^{-\ell})$ relative to the $L^2$ norm defined by $h(t)$ and Riemannian volume form, and put
\begin{equation}\label{e117}
\rho_{t,\ell}(x)=\sum_{k=0}^{N_{\ell}}|s_{t,\ell,k}|_{h(t)}^2(x),\,\forall x\in M.
\end{equation}

Inspired by \cite{Ti12} \cite{Ti13} \cite{DoSu12}, we can have the following extension of the partial $C^0$-estimates to K\"ahler-Ricci flows.

\begin{theo}
\label{th:partial-1}
Suppose $(M,g(t_i))\stackrel{d_{GH}}{\longrightarrow}(M_\infty,g_\infty)$ as phrased in Theorem \ref{regularity:1}, then
\begin{equation}\label{e118}
\inf_{t_i}\inf_{x\in M}\rho_{t_i,\ell}(x)>0
\end{equation}
for a sequence of $\ell\rightarrow\infty$.
\end{theo}

In the proof of Theorem \ref{th:partial-1}, two ingredients are important: Gradient estimate to pluri-anti-canonical sections and H\"ormander's $L^2$ estimate to $\bar{\partial}$-operator on (0,1)-forms. When Ricci curvature is bounded below, these estimates are standard and well-known, cf. \cite{Ti90} and \cite{DoSu12}.
In our case of K\"ahler-Ricci flow, the arguments should be modified because of the lack of Ricci curvature bound.

At any time $t$, for any holomorphic section $\sigma\in H^0(M,K_M^{-\ell})$, we have
\begin{equation}\label{e18}
\triangle|\sigma|^2=|\nabla\sigma|^2-n\ell|\sigma|^2
\end{equation}
and the Bochner formula
\begin{equation}
\triangle|\nabla\sigma|^2=|\nabla\nabla\sigma|^2+|\bar{\nabla}\nabla\sigma|^2-(n+2)\ell|\nabla\sigma|^2+\langle Ric(\nabla\sigma,\cdot),\nabla\sigma\rangle.\label{e19}
\end{equation}
In view of Ricci potentials, above formula can be rewritten as
\begin{equation}
\triangle|\nabla\sigma|^2=|\nabla\nabla\sigma|^2+|\bar{\nabla}\nabla\sigma|^2
-\big((n+2)\ell-1\big)|\nabla\sigma|^2-\langle\partial\bar{\partial}u(\nabla\sigma,\cdot),\nabla\sigma\rangle.\label{e110}
\end{equation}

Recall that by \cite{ZhQ} or \cite{Ye07}, there is a uniform bound of the Sobolev constant along the K\"ahler-Ricci flow. This makes it possible to
apply the standard iteration arguments of Nash-Moser to the above equations on $\sigma$ and $\nabla\sigma$.
In order to do this, we need to deal with the extra and bad term $\langle\partial\bar{\partial}u(\nabla\sigma,\cdot),\nabla\sigma\rangle$ in the iteration process by using an integration by parts and then applying Perelman's gradient estimate on $u$.
Then we can conclude the following $L^\infty$ estimate and gradient estimate on $\sigma$.

\begin{lemm}\label{lemm:gradient}
There exist constant $C=C(g_0)$ such that for any $\ell\geq 1$, $t\geq 0$, and $\sigma\in H^0(M,K_M^{-\ell})$, we have
\begin{equation}\label{gradient estimate:holomorphic section}
\|\sigma\|_{C^0}+\ell^{-\frac{1}{2}}\|\nabla\sigma\|_{C^0}\,\le\, C\ell^{\frac{n}{2}}\bigg(\int_M|\sigma|^2dv\bigg)^{1/2}.
\end{equation}
\end{lemm}

The $L^2$ estimate to the $\bar{\partial}$ operator is established firstly for K\"ahler-Einstein surfaces in \cite{Ti90}. The following is a similar estimate for the K\"ahler-Ricci flow.

\begin{lemm}\label{lemm:L2}
There exists $\ell_0$ depending on $g_0$ such that for any $\ell\geq \ell_0$, $t\geq 0$ and $\sigma\in C^\infty(M,T^{0,1}M\otimes K_M^{-\ell})$ with $\bar{\partial}\sigma=0$, we can find a solution $\bar{\partial} \vartheta=\sigma$ which satisfies
\begin{equation}\label{L2 estimate}
\int_M|\vartheta|^2dv\,\le\,4\ell^{-1}\int_M|\sigma|^2dv.
\end{equation}
\end{lemm}
\begin{proof}
It suffices to show that the Hodge Laplacian $\triangle_{\bar{\partial}}=\bar{\partial}\bar{\partial}^*+\bar{\partial}^*\bar{\partial}\geq\frac{\ell}{4}$ as an operator on $C^\infty(M,T^{0,1}M\otimes K_M^{-\ell})$ when $\ell$ is sufficiently large. Actually, this implies (i) $H^{0,1}(M,K_M^{-\ell})=0$ and thus $\bar{\partial} \vartheta=\sigma$ is solvable when $\bar{\partial}\sigma=0$ and (ii) the first positive eigenvalue of $\triangle_{\bar{\partial}}$ on $C^\infty(M,K_M^{-\ell})$ is $\geq\frac{\ell}{4}$ so that (\ref{L2 estimate}) holds for some solution $\vartheta$.

The following Weitzenb\"{o}ch type formulas hold for any $\sigma\in C^\infty(M,T^{0,1}M\otimes K_M^{-\ell})$,
\begin{equation}\label{e114}
\triangle_{\bar{\partial}}\sigma\,=\,\bar{\nabla}^*\bar{\nabla}\sigma+Ric(\sigma,\cdot)+\ell\sigma,
\end{equation}
\begin{equation}\label{e115}
\triangle_{\bar{\partial}}\sigma\,=\,\nabla^*\nabla\sigma-(n-1)\ell\sigma.
\end{equation}
A combination gives
\begin{equation}\label{e116}
\triangle_{\bar{\partial}}\sigma\,=\,(1-\frac{1}{2n})\bar{\nabla}^*\bar{\nabla}\sigma+\frac{1}{2n}\nabla^*\nabla\sigma
+(1-\frac{1}{2n})Ric(\sigma,\cdot)+\frac{\ell}{2}\sigma.
\end{equation}
Multiplying with $\sigma$ and integrating over $M$, we obtain
\begin{eqnarray}
\int\langle\triangle_{\bar{\partial}}\sigma,\sigma\rangle&=&\int\big((1-\frac{1}{2n})|\bar{\nabla}\sigma|^2+\frac{1}{2n}|\nabla\sigma|^2
+\frac{\ell}{2}|\sigma|^2\big)\nonumber\\
&&+(1-\frac{1}{2n})\int\big(|\sigma|^2-\langle\nabla\bar{\nabla}u(\sigma,\cdot),\sigma\rangle\big),\nonumber
\end{eqnarray}
where the bad term $\int\langle\nabla\bar{\nabla}u(\sigma,\cdot),\sigma\rangle$ can be estimated as follows
\begin{eqnarray}
\int\langle\nabla\bar{\nabla}u(\sigma,\cdot),\sigma\rangle&=&-\int\bar{\nabla}u\big(\langle\nabla\sigma,\sigma\rangle
+\langle\sigma,\bar{\nabla}\sigma\rangle\big)\nonumber\\
&\leq&\frac{1}{2n}\int\big(|\nabla\sigma|^2+|\bar{\nabla}\sigma|^2\big)+C\int|\sigma|^2\nonumber
\end{eqnarray}
where $C$ depend on $n$ and $\|\nabla u\|_{C^0}$. Thus,
$$\int\langle\triangle_{\bar{\partial}}\sigma,\sigma\rangle\,\ge\,\big(\frac{\ell}{2}-C\big)\int|\sigma|^2,\,\forall\sigma\in
C^\infty(M,T^{0,1}M\otimes K_M^{-\ell}).$$
In particular, $\triangle_{\bar{\partial}}\geq\frac{\ell}{4}$ when $\ell$ is large enough.
\end{proof}

The partial $C^0$ estimate for K\"ahler-Ricci flow will follow from a parallel argument as one did in the K\"ahler-Einstein case
in \cite{Ti12}, \cite{Ti13} and \cite{DoSu12}. We will adopt the notations and follow the arguments in \cite{Ti13}.

According to our results of Section 3, for any $r_j\mapsto 0$, by taking a subsequence if necessary,
we have a tangent cone ${\cal C}_x$ of $(M_\infty,\omega_\infty)$ at $x$, where ${\cal C}_x$ is the limit $\lim_{j\to \infty} (M_\infty, r_j^{-2} \omega_\infty, x)$ in the Gromov-Hausdorff topology, satisfying:

\vskip 0.1in
\noindent
${\bf TZ}_1$. Each ${\cal C}_x$ is regular outside a closed subcone ${\cal S}_x$ of complex codimension at least $2$.
Such a ${\cal S}_x$ is the singular set of ${\cal C}_x$;

\vskip 0.1in
\noindent
${\bf TZ}_2$. There is a natural K\"ahler Ricci-flat metric $g_x$ on ${\cal C}_x \backslash {\cal S}_x$ which is also a cone metric. Its K\"ahler form $\omega_x$ is equal to $\sqrt{-1} \,\partial\bar\partial \rho_x^2$ on the regular part of ${\cal C}_x$, where $\rho_x$ denotes the distance function from the vertex of ${\cal C}_x$, denoted by $x$ for simplicity.
\vskip 0.1in
We will denote by $L_x$ the trivial bundle ${\cal C}_x\times \CC$ over ${\cal C}_x$ equipped with
the Hermitian metric $e^{-\rho_x^2}\,|\cdot |^2$. The curvature of this Hermitian metric is given by $\omega_x$.

Without loss of generality, we may assume that for each $j$, $r_j^{-2}= k_j$ is an integer.

For any $ \epsilon \,>\,0$, we put
$$V(x; \epsilon)\,=\,\{ \,y \,\in \, {\cal C}_x \,|\, y\,\in\, B_{\epsilon^{-1}}(0,g_x)\,\backslash \,\overline{B_{\epsilon}(0,g_x)},\,\, d(y, {\cal S}_x )\, > \,\epsilon\,\,\},$$
where $B_R(o,g_x)$ denotes the geodesic ball of $({\cal C}_x, g_x)$ centered at the vertex and with radius $R$.

For any $\epsilon > 0$, whenever $j$ is sufficiently large, there are diffeomorphisms
$$\phi_j: V(x;\frac{\epsilon}{4})\mapsto  M_\infty\backslash {\cal S}$$
satisfying:
\vskip 0.1in
\noindent
(1) $d(x, \phi_j(V(x; \epsilon))) \,<\, 10 \epsilon r_j$ and $\phi_j(V(x;\epsilon)) \subset B_{(1+\epsilon^{-1}) r_j}(x)$, where
$B_R(x)$ the geodesic ball of $(M_\infty, \omega_\infty)$ with radius $R$ and center at $x$;
\vskip 0.1in
\noindent
(2) If $g_\infty$ is the K\"ahler metric with the K\"ahler form $\omega_\infty$ on $M_\infty\backslash {\cal S}$, then
\begin{equation}
\label{eq: bound-1}
\lim_{j\to \infty} ||r_j^{-2} \phi_j^*g_\infty - g_x ||_{C^6(V(x; \frac{\epsilon}{2}))} \,=\,0,
\end{equation}
where the norm is defined in terms of the metric $g_x$.

\begin{lemm}
\label{lemm:par-1}
For any $\delta$ sufficiently small, there are a sufficiently large $\ell \,=\,k_j$ and an isomorphism $\psi $ from the trivial bundle ${\cal C}_x\times \CC$ onto
$K_{M_\infty}^{- \ell}$
over $V(x;\epsilon)$ commuting with $\phi\,=\,\phi_j$ satisfying:
\begin{equation}
\label{eq:est-2}
|\psi(1)|_\infty^{2} \,=\, e^{-\rho_x^2} ~~~{\rm and}~~~ ||\nabla \psi ||_{C^4(V(x; \epsilon))} \,\le \, \delta,
\end{equation}
where $|\cdot|_\infty^2$ denotes the induced norm on $K_{M_\infty}^{-\ell}$ by $e^{-\frac{1}{n} u_\infty} g_\infty$,
$\nabla$ denotes the covariant derivative with respect to the norms
$|\cdot|_\infty^2$ and $e^{-\rho_x^2}\, |\cdot |^2$.
\end{lemm}

We refer the readers to \cite{Ti13} for its proof. Actually, it is easier in our case here since the singularity
${\cal S}_x$ is of complex codimension at least $2$.

Let $\epsilon\,>\,0$ and $\delta\,>\,0$ be sufficiently small and be determined later. Choose $\ell$, $\phi$ and $\psi$ as in Lemma \ref{lemm:par-1}, then
there is a section
$\tau \,=\, \psi( 1 )$ of $K_{M_\infty}^{-\ell}$ on $\phi(V(x; \epsilon))$ satisfying:
$$|\tau|_\infty^2\,=\,e^{-\rho_x^2}.$$
By Lemma \ref{lemm:par-1}, for some uniform constant $C$, we have
$$|\bar\partial \tau|_\infty \, \le C \,\delta . $$

Since ${\cal S}_x$ has codimension at least $4$, we can easily construct a smooth function $\gamma_{\bar\epsilon}$ on ${\cal C}_x$
for each $\bar \epsilon \,>\,0$ with properties: $\gamma_{\bar\epsilon }(y)\,=\,1$ if $d(y,{\cal S}_x)\,\ge\,\bar\epsilon$,
$0\,\le\,\gamma_{\bar\epsilon} \,\le\,1$, $\gamma_{\bar\epsilon} (y)\,=\,0$ in an neighborhood of ${\cal S}_x$ and
$$\int_{B_{{\bar\epsilon}^{-1}}(o,g_x)} \,|\nabla \gamma_{\bar\epsilon}|^2\, \omega_x^n\,\le\,\bar\epsilon.$$
Moreover, we may have $|\nabla \gamma_{\bar\epsilon}|\,\le\, C$ for some constant $C\,=\,C(\bar\epsilon)$.

We define for any $y\,\in \,V(x;  \epsilon)$
$$
\tilde \tau (\phi(y))\,=\,\eta(2 \delta \rho_x(y)) \, \gamma_{\bar\epsilon}(y)
\,\tau (\phi(y)).
$$
where $\eta$ is a cut-off function satisfying:
$$\eta(t)\,= \,1~~{\rm for}~~t\,\le\, 1,~~\eta(t)\,=\,0~~{\rm for}~~ t\,\ge\, 2~~{\rm and}~~|\eta'(t)|\,\le\, 1.$$

Choose $\bar\epsilon$ such that $V(x; \epsilon)$ contains the support of $\gamma_{\bar \epsilon}$.
and $\gamma_{\bar \epsilon} \,=\,1$ on $V(x; \delta_0)$, where $\delta_0\, >\,0$ is determined later.

It is easy to see that $\tilde \tau$ vanishes outside $\phi(V(x; \epsilon))$, so it extends to a smooth section of $K_{M_\infty}^{-\ell}$ on
$M_\infty$. Furthermore, $\tilde \tau $ satisfies:
\vskip 0.1in
\noindent
(i) $\tilde \tau \,=\, \tau $ on $\phi(V(x; \delta_0))$;

\vskip 0.1in
\noindent
(ii) There is an $\nu\,=\,\nu(\delta,\epsilon)$ such that
$$\int_{M_\infty} |\bar\partial \tilde \tau|_\infty^2\, \omega_\infty^n \,\le \, \nu\, r^{2n-2}.$$
Note that we can make $\nu$ as small as we want so long as $\delta$, $\epsilon$ and $\bar \epsilon$ are sufficiently small.
\vskip 0.1in

Since $(M,g(t_i))$ converge to $(M_\infty, g_\infty)$ and
the Hermitian metrics $h(t_i)$ on $K_M^{-\ell}$ converge to $h_\infty$ on $M_\infty\backslash {\cal S}$ in the $C^\infty$-topology.
Therefore, there are diffeomorphisms
$$\tilde\phi_i : M_\infty\backslash {\cal S}\,\mapsto\, M $$
and smooth isomorphisms
$$F_i:K_{M_\infty}^{-\ell}\,\mapsto\, K_M^{-\ell} $$
over $M$, satisfying:

\vskip 0.1in
\noindent
${\bf C}_1$:  $\tilde\phi_i(M_\infty\backslash N_{1/i}({\cal S}))\,\subset \,M$, where $N_{\varepsilon}({\cal S})$ is the $\varepsilon$-neighborhood of ${\cal S}$;
\vskip 0.1in
\noindent
${\bf C}_2$:  $\pi_i \circ F_i\,=\, \tilde\phi_i\circ \pi_\infty$, where $\pi_i$ and $\pi_\infty$ are corresponding projections;

\vskip 0.1in
\noindent
${\bf C}_3$: $||\tilde\phi_i^*g(t_i) - g_\infty||_{C^2(M_\infty\backslash T_{1/i}({\cal S}))}\,\to\,0$ as $i\to \infty$;

\vskip 0.1in
\noindent
${\bf C}_4$: $||F_i^* h(t_i) - h_\infty||_{C^4(M_\infty\backslash T_{1/i}({\cal S}))} \,\to \, 0$ as $i\to \infty$.

\vskip 0.1in
Put $\tilde\tau_i\,=\,F_i(\tilde\tau)$, then we deduce from the above
\vskip 0.1in
\noindent
(i) $\tilde \tau_i \,=\, F_i(\tau) $ on $\tilde\phi_i (\phi(V(x; \delta_0)))$;

\vskip 0.1in
\noindent
(ii) For $i$ sufficiently large, we have
$$\int_{M} |\bar\partial \tilde \tau_i|_i^2\, dV_{g(t_i)} \,\le \, 2 \nu\, r^{2n-2},$$
where $|\cdot|_i$ denotes the Hermitian norm corresponding to $h(t_i)$.
\vskip 0.1in

By the $L^2$-estimate in Lemma \ref{lemm:L2}, we get a section $v_i$ of $K_{M}^{-\ell }$ such that
$$\bar\partial v_i \,=\, \bar\partial \tilde \tau_i$$
and
$$\int_{M_\infty} |v_i|_{i}^2 \,dV_{g(t_i)} \,\le \, \frac{1}{\ell} \int_{M} |\bar\partial \tilde \tau_i|_{i}^2 \,dV_{g(t_i)} \,\le\, 3 \nu\, r^{2n}  .$$

Put $\sigma_i \,=\, \tilde \tau_i \,- \,v_i$, it is a holomorphic section of $K_{M}^{-\ell}$.
One can show the $C^4$-norm of $\bar\partial v_i$ on $\tilde\phi_i(\phi(V(x; \delta_0)))$ is bounded from above by
$c \delta $ for a uniform constant $c$. By the standard elliptic estimates, we have
$$\sup _{\tilde\phi(\phi(V(x; 2\delta_0)\cap B_1(o,g_x)))} |v_i|_{i}^2 \,\le \, C \,(\delta_0 r)^{-2n}\, \int_{M_i} |v_i|_i^2\, dV_{g(t_i)}
\,\le\, C \,\delta_0 ^{-2n} \,\nu .$$
Here $C$ denotes a uniform constant. For any given $\delta_0$, if $\delta$ and $\epsilon$ are sufficiently small, then we can make $\nu$ such that
$$ 8 C\,\nu\,\le\, \delta_0^{2n}.$$
It follows
$$|\sigma_i|_i \,\ge\, |F_i(\tau)|_i\,-\,|v_i|_i\,\ge \,\frac{1}{2}~~~{\rm on}~~\tilde\phi_i(\phi(V(x; \delta_0)\cap B_1(o,g_x))).$$
On the other hand, by applying Lemma \ref{lemm:gradient} to $\sigma_i$, we get
$$\sup_{M} |\nabla \sigma_i|_i \,\le\, C'
\ell ^{\frac{n+1}{2}} \left (\int_{M} |\sigma_i|_i^2\,dV_{g(t_i)}\right )^{\frac{1}{2}} \,\le\, C'\, r^{-1} .$$
Since the distance $d(x, \phi(\delta_0 u)) $ is less than $ 10 \delta_0 r$ for some $u \,\in\,\partial B_1(o,g_x)$, if $i$ is sufficiently large,
we deduce from the above estimates
$$|\sigma_i|_i (x_i)\, \ge\, 1/4 - C'\,\delta_0,$$
hence, if we choose $\delta_0$ such that $C' \delta_0 < 1/8$, then $\rho_{\omega_i,\ell} (x_i) > 1/8$.

Theorem \ref{th:partial-1}, i.e., the partial $C^0$-estimate
for $g(t_i)$ in the K\"ahler-Ricci flow, is proved.

Using the same arguments as those in proving Theorem 5.9 in \cite{Ti13}, we can deduce Theorem \ref{regularity:2} from Theorem \ref{th:partial-1}.

\section{A corollary of Conjecture \ref{conj:HT}}

In this last section, we will show how to deduce the Yau-Tian-Donaldson conjecture in case of Fano manifolds from Conjecture \ref{conj:HT}. The key is to prove
that there is an uniform lower bound for Mabuchi's $K$-energy along the K\"ahler-Ricci flow provided the partial $C^0$ estimate and $K$-stability of the manifold.

Let $\omega(t)$ be the K\"ahler form of the K\"ahler-Ricci flow $g(t)$. For K\"ahler metrics $\omega_1,\omega_2\in2\pi c_1$, denote by $K(\omega_1,\omega_2)$
the relative Mabuchi's $K$-energy from $\omega_1$ to $\omega_2$ (the function $M$ in \cite{Ma86}).

\begin{theo}\label{K-energy}
Suppose the partial $C^0$ estimate {\rm(\ref{e118})} holds for a sequence of times $t_i\rightarrow\infty$. If $M$ is $K$-stable, then the $K$-energy is bounded below under the K\"ahler-Ricci flow
\begin{equation}
K(\omega(0),\omega(t))\,\ge\, -C(g_0).
\end{equation}
\end{theo}
\begin{proof}
It is well known that $K(\omega(0),\omega(t))$ is non-increasing in $t$ (cf. \cite{TiZhu07}). So it suffices to show a uniform lower bound of $K(\omega(0),\omega_i)$ where $\omega_i=\omega(t_i)$. We will establish this by using a result of S. Paul. It is proved in \cite{Pa12, Pa12b} that if $M$ is $K$-stable, then the $K$-energy
is bounded from below on the space of Bergman-type metrics which arise from the Kodaira embedding via bases of $K_M^{-\ell}$.

Fix an integer $\ell>0$ sufficiently large such that $K_M^{-\ell}$ is very-ample and $M$ is $K$-stable with respect to $K_M^{-\ell}$. Any orthonormal basis $\{s_{t_i,\ell,k}\}_{k=0}^{N_{\ell}}$ of $H^0(M,K_M^{-\ell})$ at $t_i$ defines an embedding
$$\Phi_i:M\rightarrow\mathbb{C}P^{N_\ell}.$$
Let $\omega_{FS}$ be the Fubini-Study metric on $\mathbb{C}P^{N_\ell}$ and put $\tilde{\omega}_i=\frac{1}{\ell}\Phi_i^*\omega_{FS}$, the Bergman metric associated to $\Phi_i$. For any $i\ge 1$, there exists a $\sigma_i\in SL(N_\ell+1,\mathbb{C})$ such that $\Phi_i=\sigma_i\circ \Phi_1$. By the result of \cite{Pa12b}, we have 
$$K(\tilde{\omega}_1,\tilde{\omega}_i)\,\ge\,-C, $$
where $C$ is a uniform constant. By the cocycle condition of the $K$-energy,
$$K(\omega(0),\omega_i)+K(\omega_i,\tilde{\omega}_i)\,=\,K(\omega(0),\tilde{\omega}_i)\,=\,K(\omega(0),\tilde{\omega}_1)
+K(\tilde{\omega}_1,\tilde{\omega}_i)\,\ge\, -C.$$
Therefore, to show that $K(\omega(0),\omega_i)$ is bounded from below, we only need to get an upper bound for $K(\omega_i,\tilde{\omega}_i)$.

Put $\tilde{\rho}_i=\frac{1}{\ell}\rho_{t_i,\ell}$, where $\rho_{t_i,\ell}$ is defined by (\ref{e117}) with $t=t_i$. Then
$$\omega_i\,=\,\tilde{\omega}_i\,+\,\sqrt{-1}\partial\bar{\partial}\,\tilde{\rho}_i.$$
The $K$-energy has the following explicit expression \cite{Ti},
\begin{eqnarray}
K(\omega_i,\tilde{\omega}_i)&=&\int_M\log\frac{\tilde{\omega}_i^n}{\omega_i^n}\tilde{\omega}_i^n
+\int_Mu(t_i)\big(\tilde{\omega}^n_i-\omega_i^n\big)\nonumber\\
&&\,-\sum_{k=0}^{n-1}\frac{n-k}{n+1}\int_M\sqrt{-1}\partial\tilde{\rho}_i\wedge\bar{\partial}\tilde{\rho}_i
\wedge\omega_i^k\wedge\tilde{\omega}_i^{n-k-1},\nonumber
\end{eqnarray}
where $u(t_i)$ is the Ricci potential at time $t_i$ of the K\"ahler-Ricci flow. Thus,
$$K(\omega_i,\tilde{\omega}_i)\,\le\,\int_M\log\frac{\tilde{\omega}_i^n}{\omega_i^n}\tilde{\omega}_i^n
+\int_Mu(t_i)\big(\tilde{\omega}^n_i-\omega_i^n\big).$$
By Perelman's estimate, we have $|u(t_i)|\le C(g_0)$. It follows
$$K(\omega_i,\tilde{\omega}_i)\,\le\,\int_M\log\frac{\tilde{\omega}_i^n}{\omega_i^n}\tilde{\omega}_i^n+C.$$
Finally, by using the partial $C^0$-estimate and applying the gradient estimate in Lemma \ref{lemm:gradient} to each $s_{t_i,\ell,k}$,
we have
$$\tilde{\omega}_i\,\le\, C(g_0)\cdot\omega_i.$$
This gives a desired upper bound of $K(\omega_i,\tilde{\omega}_i)$, and consequently, a lower bound of $K(\omega(0),\omega_i)$. The proof is now completed.
\end{proof}

Theorem \ref{K-energy} implies that the limit $M_\infty$ must be K\"ahler-Einstein (see \cite{TiZhu07} for example). Then its automorphism group
must be reductive as a corollary of the uniqueness theorem due to B. Berndtsson and R. Berman (see \cite{Be13}). It follows that if $M_\infty$ is not equal to $M$,
there is a $\CC^*$-action $\{\sigma(s)\}_{s\in \CC^*}\subset SL(N_\ell+1,\CC)$ such that $\sigma(s)\cdot\Phi_1(M)$ converges to the embedding of $M_\infty$ in $\CC P^{N_\ell}$.
This contradicts to the K-stability since the Futaki invariant of $M_\infty$ vanishes. Hence, there is a K\"ahler-Einstein metric on $M=M_\infty$.

\begin{rema}
In fact, using a very recent result of S. Paul \cite{Pa12c} and the same argument as those in the proof of Theorem \ref{K-energy},
we can prove directly that the K-energy is proper along the K\"ahler-Ricci flow, so
the flow converges to a K\"ahler-Einstein metric on the same underlying K\"ahler manifold.
\end{rema}

As a final remark, we outline a method of directly producing a non-trivial holomorphic vector field on $M_\infty$ if it is different from $M$.
Suppose $M_\infty$ is not isomorphic to $M$. Let $\lambda(t)$ be the smallest eigenvalue of the weighted Laplace $\triangle_u=\triangle-g^{i\bar{j}}\partial_iu\partial_{\bar{j}}$ at time $t$, where $u=u(t)$ is the Ricci potential of $g(t)$ defined in (\ref{Ricci potential}).
The Poincar\'e inequality \cite{Fu} shows $\lambda(t)> 1$. According to Theorem 1.5 of \cite{Zh11}, $\lambda(t_i)\rightarrow 1$ as $i\rightarrow\infty$.
If we denote by $\theta(t_i)$ an eigenfunction of $\lambda(t_i)$ satisfying the normalization:
$$\int|\theta_i|^2e^{-u(t_i)}dv_{g(t_i)}=1,$$
then by the Nash-Moser iteration, we have the following gradient estimate:

\begin{lemm}
There exists $C=C(g_0)$ such that any eigenfunction $\theta$, at any time $t$, satisfying
\begin{equation}
g^{i\bar{j}}\big(\partial\partial_{\bar{j}}\theta-\partial_iu\partial_{\bar{j}}\theta\big)\,=\,\lambda\theta
\end{equation}
has the gradient estimates
\begin{equation}
\|\bar{\partial}\theta\|_{C^0}+\|\partial\theta\|_{C^0}\,\le\, C\lambda^{\frac{n+1}{2}}\|\theta\|_{L^2}.
\end{equation}
\end{lemm}
It follows that $\theta(t_i)$ converges to a nontrivial eigenfunction $\theta_\infty$ with eigenvalue $1$ on the limit variety $M_\infty$. By an easy calculation,
$$\int_M|\bar{\nabla}\bar{\nabla}\theta_i|^2e^{-u(t_i)}dv_{g(t_i)}\,=\,\lambda(t_i)\big(\lambda(t_i)-1\big)\rightarrow 0.$$
Together with Perelman's $C^0$ estimate on $u$, this yields a bounded holomorphic vector field on $M_\infty$ as the gradient field of $\theta_\infty$.

\end{document}